\documentclass[letterpaper, 10pt, journal, onecolumn, final]{IEEEtran}
\usepackage{cite}
\usepackage{algorithm}
\usepackage{algpseudocode}
\usepackage{algorithmicx}
\usepackage{graphicx}
\usepackage{textcomp}
\usepackage{amssymb,amsmath,amsfonts,cases,mathtools}
\usepackage{microtype}
\usepackage{url}
\usepackage{hyperref}
\usepackage{xifthen}
\usepackage{xpatch}
\usepackage{amsthm}
\usepackage[dvipsnames]{xcolor}
\usepackage{tikz}
\usepackage{lipsum,booktabs}
\usepackage{pifont}
\usepackage{xcolor}
\usepackage{soul}
\usepackage{balance}

\newcommand{\R}{\mathbb{R}}

\newcommand{\N}{\mathbb{N}}

\newcommand{\T}{^\top}
\newcommand{\inv}{^{-1}}
\newcommand{\until}[1]{\{1,\ldots,#1\}} 

\newcommand{\argmin}{\mathop{\rm argmin}}

\newcommand{\subj}{\textnormal{subj.~to}}

\DeclareMathOperator{\coltext}{col}

\newtheorem{theorem}{Theorem}
\newtheorem{proposition}{Proposition}

\newtheorem{definition}{Definition}

\newtheorem{assumption}{Assumption}
\newtheorem{remark}{Remark}

\newcommand\oprocendsymbol{\hbox{$\blacksquare$}}
\newcommand\oprocend{\relax\ifmmode\else\unskip\hfill\fi\oprocendsymbol}

\newcommand{\VI}[1]{\text{VI}(#1)}

\newcommand{\cC}{\mathcal C}

\newcommand{\cE}{\mathcal E}

\newcommand{\cG}{\mathcal G}
\newcommand{\cH}{\mathcal H}
\newcommand{\cI}{\mathcal I}

\newcommand{\cL}{\mathcal L}
\newcommand{\cM}{\mathcal M}
\newcommand{\cN}{\mathcal N}

\newcommand{\cS}{\mathcal S}
\newcommand{\cT}{\mathcal T}

\newcommand{\cW}{\mathcal W}
\newcommand{\cX}{\mathcal X}

\graphicspath{{figs/}}

\newcommand{\norm}[1]{\left \|#1 \right \|}
\newcommand{\col}[1]{\coltext\left[#1\right]}

\newcommand{\rz}{\mathrm{z}}

\newcommand{\rx}{\mathrm{x}}
\newcommand{\rw}{\mathrm{w}}
\newcommand{\rs}{\mathrm{s}}

\newcommand{\z}{\rz}

\newcommand{\x}{\rx}
\newcommand{\w}{\rw}
\newcommand{\s}{\rs}

\newcommand{\ud}{^}

\newcommand{\iter}{t}
\newcommand{\iterp}{{\iter+1}}

\newcommand{\initer}{\tau}
\newcommand{\initerp}{{\initer+1}}

\newcommand{\n}{n}
\newcommand{\m}{m}

\newcommand{\xt}{\x\ud\iter}

\newcommand{\zt}{\z\ud\iter}
\newcommand{\ztp}{\z\ud{\iterp}}

\newcommand{\bz}{\bar{\z}}
\newcommand{\bzt}{\bz\ud\iter}
\newcommand{\bztp}{\bz\ud\iterp} 

\newcommand{\pz}{\z_\perp} 
\newcommand{\pzt}{\pz\ud\iter} 
\newcommand{\pztp}{\pz\ud\iterp}

\newcommand{\lipp}{\beta}

\newcommand{\xit}{\x\ud\iter_{i}}
\newcommand{\xjt}{\x\ud\iter_{j}}
\newcommand{\xitp}{\x\ud{\iterp}_{i}}

\newcommand{\F}{F}
\newcommand{\X}{X}
\newcommand{\xstar}{x^\star}

\newcommand{\f}{f}
\newcommand{\phii}{\phi_{i}}
\newcommand{\phij}{\phi_{j}}

\newcommand{\J}{J}
\newcommand{\Ji}{\J_{i}}
\newcommand{\xmi}{x_{-i}}
\newcommand{\xstarmi}{\xstar_{-i}}
\newcommand{\xstari}{\xstar_{i}}
\newcommand{\cCi}{\cC_{i}}

\newcommand{\cmi}{c_{-i}}
\newcommand{\sigmai}{\sigma_{-i}}

\newcommand{\map}[3]{#1: #2 \rightarrow #3}

\newcommand{\1}{\mathbf{1}}

\newcommand{\lit}{\lambda\ud\iter_i}
\newcommand{\ljt}{\lambda\ud\iter_j}
\newcommand{\litp}{\lambda\ud\iterp_i}

\newcommand{\wit}{\w\ud\iter_i}
\newcommand{\wjt}{\w\ud\iter_j}
\newcommand{\witp}{\w\ud\iterp_i}

\newcommand{\zit}{\zt_i}
\newcommand{\zjt}{\zt_j}
\newcommand{\zitp}{\ztp_i}
\newcommand{\zni}{\z_{\cN_i}}
\newcommand{\znit}{\zni\ud\iter}

\newcommand{\ztt}{\z\ud{\iter,\initer}}
\newcommand{\zttp}{\z\ud{\iter,\initerp}}
\newcommand{\zitt}{\ztt_i}
\newcommand{\zjtt}{\ztt_j}
\newcommand{\zittp}{\zttp_i}
\newcommand{\znitt}{\zni\ud{\iter,\initer}}

\newcommand{\statet}{\state\ud\iter}
\newcommand{\statetp}{\state\ud\iterp}

\newcommand{\pr}{\delta}
\newcommand{\step}{\gamma}
\newcommand{\degi}{\text{deg}_i}

\newcommand{\state}{\chi}

\newcommand{\statestar}{\state^\star}

\newcommand{\cSc}{\cS_{\state}}
\newcommand{\cSz}{\cS_{\z}}

\newcommand{\cSbzI}{\cS_{\bzI}}
\newcommand{\cSbzE}{\cS_{\bzE}}

\newcommand{\cSbz}{\cS_{\bz}}

\newcommand{\lstar}{\lambda^\star}

\newcommand{\agg}{\alpha}

\newcommand{\lagg}{\varphi}
\newcommand{\laggi}{\lagg_i}

\newcommand{\statei}{\state_i}
\newcommand{\statej}{\state_j}
\newcommand{\statek}{\state_k}

\newcommand{\stateit}{\statei\ud\iter}
\newcommand{\statejt}{\statej\ud\iter}
\newcommand{\statekt}{\statek\ud\iter}
\newcommand{\stateitp}{\statei\ud\iterp}

\newcommand{\stateni}{\state_{\cN_i}}
\newcommand{\statenit}{\stateni\ud\iter}

\newcommand{\al}{g}
\newcommand{\ali}{\al_i}

\newcommand{\tr}{h}
\newcommand{\tri}{\tr_i}

\newcommand{\na}{a}

\newcommand{\pzeq}{\pz\ud{\text{eq}}}

\newcommand{\lippa}{\lipp_{\al}}

\newcommand{\nstate}{n_\state}
\newcommand{\nstatei}{n_{\state_i}}
\newcommand{\nstatej}{n_{\state_j}}

\newcommand{\tz}{\tilde{\z}_\perp}
\newcommand{\tzt}{\tz\ud\iter}
\newcommand{\tztp}{\tz\ud\iterp}

\newcommand{\Vz}{U}
\newcommand{\Vc}{W}

\newcommand{\slow}{s}
\newcommand{\fast}{f}

\newcommand{\ala}{\al_{\agg}}

\newcommand{\Xstar}{\X^\star}
\newcommand{\Xvstar}{\Xv^\star}

\newcommand{\out}{\eta}
\newcommand{\outi}{\out_i}

\newcommand{\Xv}{\X}

\newcommand{\nz}{p}
\newcommand{\nzi}{p_i}
\newcommand{\nzj}{p_j}

\newcommand{\zetat}{\zeta\ud\iter} 
\newcommand{\zetatp}{\zeta\ud\iterp}

\newcommand{\zetajt}{\zeta_j\ud\iter}
\newcommand{\zetait}{\zeta_i\ud\iter}
\newcommand{\zetaitp}{\zeta_i\ud\iterp}

\newcommand{\tagg}{\hat{\agg}}
\newcommand{\taggi}{\tagg_i}
\newcommand{\taggj}{\tagg_j}

\newcommand{\ptagg}{\tagg_\perp}
\newcommand{\ptaggI}{\tagg_{\ssI,{\scriptscriptstyle \perp}}}
\newcommand{\ptaggE}{\tagg_{\ssE,{\scriptscriptstyle \perp}}}

\newcommand{\cnsor}{consensus-oriented }

\newcommand{\ssI}{{\scriptscriptstyle I}}
\newcommand{\ssE}{{\scriptscriptstyle E}}

\newcommand{\aggI}{\agg_\ssI}
\newcommand{\laggI}{\lagg_\ssI}
\newcommand{\laggIi}{\lagg_{\ssI,\scriptscriptstyle{i}}}
\newcommand{\laggIj}{\lagg_{\ssI,\scriptscriptstyle{j}}}

\newcommand{\laggIni}{\lagg_{\ssI,\scriptscriptstyle{\cN_i}}}

\newcommand{\aggE}{\agg_\ssE}
\newcommand{\laggE}{\lagg_\ssE}
\newcommand{\laggEi}{\lagg_{\ssE,\scriptscriptstyle{i}}}
\newcommand{\laggEj}{\lagg_{\ssE,\scriptscriptstyle{j}}}

\newcommand{\naI}{\na_\ssI}
\newcommand{\naE}{\na_\ssE}

\newcommand{\zI}{\z_\ssI}
\newcommand{\zIt}{\zI\ud\iter} 
\newcommand{\zItp}{\zI\ud\iterp} 

\newcommand{\zE}{\z_\ssE}
\newcommand{\zEt}{\zE\ud\iter} 
\newcommand{\zEtp}{\zE\ud\iterp} 

\newcommand{\zIi}{\z_{\ssI,\scriptscriptstyle{i}}}

\newcommand{\zIit}{\zIi\ud\iter}

\newcommand{\zEi}{\z_{\ssE,\scriptscriptstyle{i}}}

\newcommand{\zEit}{\zEi\ud\iter}

\newcommand{\zIitp}{\zIi\ud\iterp}
\newcommand{\zEitp}{\zEi\ud\iterp}

\newcommand{\bzI}{\bz_\ssI}

\newcommand{\bzItp}{\bzI\ud\iterp}

\newcommand{\bzE}{\bz_\ssE}

\newcommand{\bzEtp}{\bzE\ud\iterp}

\newcommand{\pzI}{\z_{\ssI,{\scriptscriptstyle \perp}}}
\newcommand{\pzIt}{\pzI\ud\iter}
\newcommand{\pzItp}{\pzI\ud\iterp}

\newcommand{\pzE}{\z_{\ssE,{\scriptscriptstyle \perp}}}
\newcommand{\pzEt}{\pzE\ud\iter}
\newcommand{\pzEtp}{\pzE\ud\iterp}

\newcommand{\laggEni}{\lagg_{\ssE,{\scriptscriptstyle \cN_i}}}

\newcommand{\zIni}{\z_{\ssI,{\scriptscriptstyle \cN_i}}}

\newcommand{\zInit}{\zIni\ud\iter}

\newcommand{\taggIi}{\tagg_{\ssI,\scriptscriptstyle{i}}}
\newcommand{\taggEi}{\tagg_{\ssE,\scriptscriptstyle{i}}}

\newcommand{\taggIj}{\tagg_{\ssI,\scriptscriptstyle{j}}}

\newcommand{\taggIni}{\tagg_{\ssI,\scriptscriptstyle{\cN_i}}}

\newcommand{\taggI}{\tagg_\ssI}
\newcommand{\taggE}{\tagg_\ssE}

\newcommand{\trIi}{\tr_{\ssI,\scriptscriptstyle{i}}}
\newcommand{\trEi}{\tr_{\ssE,\scriptscriptstyle{i}}}
\newcommand{\trI}{\tr_{\ssI}}
\newcommand{\trE}{\tr_{\ssE}}

\newcommand{\nzIi}{n_{z_{\ssI,\scriptscriptstyle{i}}}}
\newcommand{\nzEi}{n_{z_{\ssE,\scriptscriptstyle{i}}}}

\newcommand{\nzIj}{n_{z_{\ssI,\scriptscriptstyle{j}}}}
\newcommand{\nzEj}{n_{z_{\ssE,\scriptscriptstyle{j}}}}

\newcommand{\nzI}{n_{\z_\ssI}}
\newcommand{\nzE}{n_{\z_\ssE}}

\newcommand{\pzIeq}{\z_{\ssI,{\scriptscriptstyle \perp}}^{\text{eq}}}
\newcommand{\pzEeq}{\z_{\ssE,{\scriptscriptstyle \perp}}^{\text{eq}}}

\newcommand{\uI}{u_\ssI}
\newcommand{\uE}{u_\ssE}

\newcommand{\tzI}{\tilde{\z}_{\ssI,{\scriptscriptstyle \perp}}}

\newcommand{\tzE}{\tilde{\z}_{\ssE,{\scriptscriptstyle \perp}}}

\newcommand{\tzIt}{\tzI\ud\iter}
\newcommand{\tzItp}{\tzI\ud\iterp}

\newcommand{\tzEt}{\tzE\ud\iter}
\newcommand{\tzEtp}{\tzE\ud\iterp}

\newcommand{\VzI}{\Vz_\ssI}
\newcommand{\VzE}{\Vz_\ssE}

\newcommand{\ttr}{\tilde{\tr}}

\newcommand{\pttrI}{\ttr_{\ssI,{\scriptscriptstyle \perp}}}
\newcommand{\pttrE}{\ttr_{\ssE,{\scriptscriptstyle \perp}}}

\newcommand{\pt}{p\ud\iter} 
\newcommand{\ptp}{p\ud\iterp} 

\newcommand{\pit}{\pt_i}
\newcommand{\pjt}{\pt_j}
\newcommand{\pitp}{\ptp_i}

\newcommand{\qt}{q\ud\iter} 
\newcommand{\qtp}{q\ud\iterp} 

\newcommand{\qit}{\qt_i}
\newcommand{\qjt}{\qt_j}
\newcommand{\qitp}{\qtp_i}

\newcommand{\zij}{\z_{ij}}
\newcommand{\zji}{\z_{ji}}

\newcommand{\zijt}{\zij\ud\iter} 
\newcommand{\zjit}{\zji\ud\iter} 

\newcommand{\zijtp}{\zij\ud\iterp} 

\newcommand{\lippeq}{\lipp_{\text{eq}}}

\newcommand{\set}{\cI}

\newcommand{\cHp}{\cH_\rho}

\newcommand{\cXs}{\cX^\star}

\newcommand{\fs}{\f_{\sigma}}

\newcommand{\nb}{\bar{\nz}} 
\newcommand{\np}{\nz_\perp} 

\newcommand{\nbI}{\nb_\ssI} 
\newcommand{\npI}{\nz_{\ssI,{\scriptscriptstyle \perp}}}

\newcommand{\nbE}{\nb_\ssE} 
\newcommand{\npE}{\nz_{\ssE,{\scriptscriptstyle \perp}}}

\newcommand{\TT}{\cT}
\newcommand{\Tb}{\bar{\TT}} 
\newcommand{\Tp}{\TT_\perp}

\newcommand{\TTi}{\TT\inv}

\newcommand{\TTI}{\TT_\ssI}
\newcommand{\TbI}{\Tb_\ssI} 
\newcommand{\TpI}{\TT_{\ssI,{\scriptscriptstyle \perp}}}

\newcommand{\TTE}{\TT_\ssE}
\newcommand{\TbE}{\Tb_\ssE} 
\newcommand{\TpE}{\TT_{\ssE,{\scriptscriptstyle \perp}}}

\newcommand{\ptr}{\tr_\perp}

\newcommand{\ptrI}{\tr_{\ssI,{\scriptscriptstyle \perp}}}
\newcommand{\ptrE}{\tr_{\ssE,{\scriptscriptstyle \perp}}}

\newcommand{\cSze}{\cS_{\zeta}}
\newcommand{\Vzeta}{\Vz}

\newcommand{\tzeta}{\tilde{\zeta}}
\newcommand{\tzetat}{\tzeta\ud\iter} 
\newcommand{\tzetatp}{\tzeta\ud\iterp}

\newcommand{\zetaeq}{\zeta\ud{\text{eq}}}

\newcommand{\uIi}{u_{\ssI,{\scriptscriptstyle i}}}
\newcommand{\uEi}{u_{\ssE,{\scriptscriptstyle i}}}

\newcommand{\const}{\nu}

\newcommand{\tfc}{\tilde{\f}}

\newcommand{\Gg}{G}
\newcommand{\Gi}{\Gg_i}

\newcommand{\cf}{c_2}
\newcommand{\cff}{c_3}
\newcommand{\cfff}{c_4}

\newcommand{\dd}{D}

\def\er/{Erd\H{o}s-R\'enyi}

\def\acro/{LCDAC}
\def\centralized/{centralized}
\def\Centralized/{Centralized}

\title{A Unifying System Theory Framework for\\ Distributed Optimization and Games} 

\author{Guido Carnevale,
Nicola Mimmo,
Giuseppe Notarstefano, 
\thanks{This work was supported in part by Fondi PNRR - Bando PE - Progetto PE11 - 3A-ITALY, ``Made in Italy Circolare e Sostenibile'' - Codice PE0000004, CUP: J33C22002950001 and the European Union - PRIN 2022 ECODREAM Energy COmmunity management: DistRibutEd AlgorithMs and toolboxes for efficient and sustainable operations code [202228CTKY 002] - CUP [J53D23000560006].
The authors are with the Department of Electrical, Electronic and Information Engineering, Alma Mater Studiorum - Universit\`{a} di Bologna, Bologna, 40136, Italy, e-mail: \{guido.carnevale, nicola.mimmo2, giuseppe.notarstefano\}@unibo.it.}}

\begin{document}
\maketitle

\begin{abstract}
	This paper introduces a systematic methodological framework to
	design and analyze distributed algorithms for optimization and games
	over networks. Starting from a centralized method, we identify an
	aggregation function involving all the decision variables (e.g., a
	global cost gradient or constraint) and introduce a distributed consensus-oriented scheme to asymptotically approximate the
	unavailable information at each agent. Then, we delineate the proper
	methodology for intertwining the identified building blocks, i.e.,
	the optimization-oriented method and the consensus-oriented one. The
	key intuition is to interpret the obtained interconnection as a
	singularly perturbed system. We rely on this interpretation to
	provide sufficient conditions for the building blocks to be
	successfully connected into a distributed scheme exhibiting the
	convergence guarantees of the centralized algorithm. Finally, we
	show the potential of our approach by developing a new distributed
	scheme for constraint-coupled problems with a linear convergence rate.
  \end{abstract}

\section{Introduction}

Increasing attention has been devoted in the last years to turning centralized (or parallel) architectures into distributed ones,
since distributed algorithms offer many benefits: e.g., preserve privacy and avoid a single point of failure, see Fig.~\ref{fig:graph}.
\begin{figure}[H]
	\centering
	\includegraphics[scale = .75]{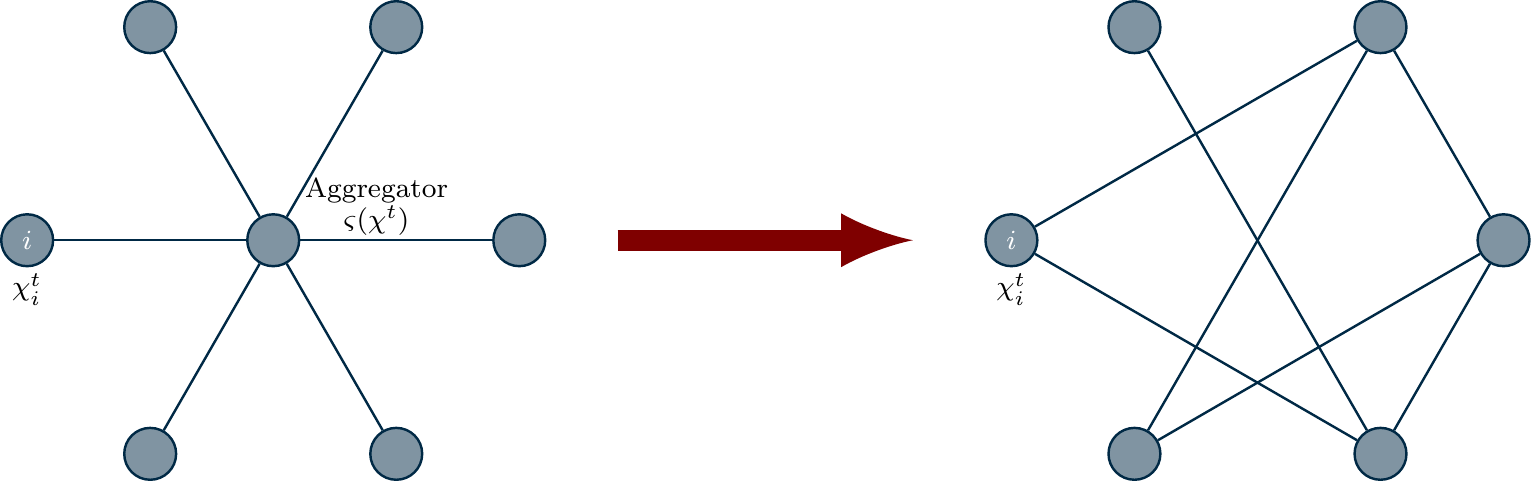}
	\caption{\Centralized/ (left) and distributed (right) architectures.}\label{fig:graph}
\end{figure}
In \emph{distributed optimization}, agents cooperate to minimize a
common performance index, while, in the context of \emph{games} over
networks, they compete with each other to minimize individual costs.
See the recent surveys on distributed
optimization~\cite{testa2023tutorial,notarstefano2019distributed,giselsson2018large,molzahn2017survey,nedic2018distributed,yang2019survey, negenborn2014distributed,necoara2011parallel} and aggregative games~\cite{jensen2010aggregative,parise2021analysis,belgioioso2022distributed}
for a comprehensive overview of the existing frameworks, algorithms, and application scenarios.
The use of tools from system theory in these settings is a recent successful trend~\cite{dorfler2024towards}.
The advantages of this approach have been already shown in the context of both distributed optimization~\cite{wang2010control,wang2011control,hatanaka2018passivity,chang2019saddle}
and games~\cite{gadjov2017continuous,gadjov2022exact}.
The work~\cite{van2022universal} (see also~\cite{van2020systematic}) proposes a control-theoretic approach to systematically decompose distributed optimization algorithms into a \centralized/
optimization method and a second-order consensus estimator.
Other systematic decompositions based on system theory are proposed
in~\cite{sundararajan2019canonical,sundararajan2020analysis}, where also time-varying graphs are considered.
A first attempt of a systematic design for distributed continuous-time optimization schemes is proposed in~\cite{touri2023unified} by relying on nonlinear observation theory.
Our key idea is to formalize an intuition of most
distributed algorithms that try to mimic a centralized algorithm by combining an approximated optimization part with a consensus scheme to agree on global information. 
Specific distributed settings have been addressed under the lens of \emph{singular perturbations} in both optimization~\cite{zanella2011newton,varagnolo2015newton,carnevale2022nonconvex,carnevale2022aggregative,maritan2023zo,carnevale2023admm} and game theory~\cite{gadjov2018passivity,carnevale2022tracking,krilavsevic2023stability,krilavsevic2023discrete,ochoa2023momentum}.
Singular perturbations (or timescale separation) are a powerful tool for analyzing systems characterized by the interconnection of a slow subsystem with a fast one~\cite{kokotovic1976singular,khalil2002nonlinear}.
In order to deal with agreement on global information in our schemes, we resort to the formalism of
$\agg$-consensus~\cite{bauso2006non,cortes2008distributed},
modeling problems in which $N$ agents endowed with local
quantities $s_1,\dots,s_N$ aim to reach consensus at $\agg(s_1,\dots,s_N)$.

The main contribution of this paper is a methodological
framework for both the design and analysis of distributed algorithms for
optimization and games over networks.
We develop a procedure that (i) systematically prescribes how to get a distributed version of an inspiring centralized method and (ii) provides analytical guarantees that the resulting distributed scheme achieves the same convergence properties of the centralized one.
Inspired by the mentioned successful examples using singular perturbations, we adopt them as a foundational tool to develop our systematic approach.
Specifically,
our procedure takes a centralized algorithm and identifies a so-called
\emph{aggregation} function, i.e., a global quantity depending on
all the network variables and, thus, requiring a central unit to be
computed.
Then, the global information is substituted by a local proxy that is
updated according to a so-called \cnsor dynamics.
Our framework formally establishes how to interlace the optimization-
and the \cnsor parts leading to a distributed algorithm with suitable
convergence guarantees.
We interpret the obtained interconnected dynamics as a Singularly
Perturbed (SP) system, usually decomposed into \emph{fast} and
\emph{slow} subsystems derived from the optimization-oriented part and the \cnsor one, respectively.
Based on this structure and by using tools from system theory, we
provide convergence guarantees for the resulting distributed algorithm based on proper conditions on the optimization-oriented part and the \cnsor one separately considered.
We consider general cases, possibly involving multiple solutions
(e.g., nonconvex scenarios) and guarantee convergence in a
LaSalle sense.
When the solution is unique, we establish conditions for its
exponential stability or, equivalently, for linear convergence.
Our main result stands on a general theorem for SP systems with LaSalle-type properties (Theorem~\ref{th:generic} in Appendix~\ref{sec:SP})
extending~\cite[Th.~1]{carnevale2022nonconvex} and, thus, representing
a contribution per se.
Another side contribution is Proposition~\ref{prop:nested_trackers} (in
Appendix~\ref{sec:nested}): it establishes sufficient conditions for
designing distributed schemes solving dynamic average
consensus problems with composite functions.
Finally, we showcase the potential of the proposed approach with a concrete example.
Specifically, we design a novel distributed algorithm for constraint-coupled
optimization scenarios and, under mild problem assumptions, we prove
its linear convergence properties.
To the best of our knowledge, such a novel algorithm is the first distributed scheme exhibiting a linear
convergence rate in solving optimization problems with coupling \emph{inequality} constraints.
This example highlights how novel distributed algorithms can be designed thanks to our framework by means of simple steps.
Thus, it shows that our approach paves the way for systematically obtaining a wide variety of distributed algorithms (and their analysis) across several existing and potentially new scenarios.
As for the existing attempts for systematic approaches (see the mentioned references~\cite{van2022universal,van2020systematic,sundararajan2019canonical,sundararajan2020analysis,touri2023unified}), we note that they all restrict to unconstrained convex consensus optimization and all but~\cite{touri2023unified} focus only on the analysis neglecting the design phase.
The work~\cite{touri2023unified} offers also design arguments but it is restricted to gradient-flow continuous-time schemes.
Let us then summarize the peculiarities of our methodology.
First, it offers not only the analysis but also a systematic algorithm design.
Second, thanks to the inherent generality of system theory, it provides a unified perspective across different viewpoints.
Indeed, it allows us to encompass a wide range of scenarios (e.g., constraint-coupled optimization or aggregative games), algorithms (e.g., augmented primal-dual or projected gradient methods), and assumptions (e.g., nonconvex costs or unbounded feasible sets).

The paper is organized as follows.
In Section~\ref{sec:problem}, we formulate the problem and present the more popular frameworks arising in distributed optimization and games.
Section~\ref{sec:systematic_design} presents our systematic approach and the main result of the paper.
Section~\ref{sec:overview_of_algorithms} provides an overview of optimization- and \cnsor algorithms matching the assumptions made in Section~\ref{sec:systematic_design}.
Finally, in Section~\ref{sec:application}, we apply our approach to synthesize a novel distributed scheme and numerically test its effectiveness.

\paragraph*{Notation}

$\col{x_1,\dots,x_N}$ denotes the column stacking of the vectors $x_1,\dots,x_N$.
In $\R^{m\times m}$, $I_m$ and $0_m$ are the identity and zero matrices. 
The symbol $1_N$ denotes the vector of $N$ ones while $\1_{N,d} := 1_N \otimes I_d$, where $\otimes$ denotes the Kronecker product. 
Dimensions are omitted when they are clear from the context.
Given $v \!\in\! \R^n$, we denote as $[v]_i$ its $i$-th component. %
Given $x \in \R^{n}$ and $\cS \subset \R^{\n}$, we define $\norm{x}_{\cS}:= \inf_{y \in \cS}\norm{x - y}$.
Given $f: \R^{n_1} \! \times \! \R^{n_2} \! \to \! \R$, we define $\nabla_1 f(x,y) \!:=\! \tfrac{\partial}{\partial s}f(s,y)|_{s = x}\T$ and $\nabla_2 f(x,y) \!:= \!\tfrac{\partial}{\partial s}f(x,s)|_{s = y}\T$. 
$\R_{+}^n$ is the positive orthant in $\R^n$.

\section{Problem Formulation and Preliminaries}
\label{sec:problem}

Our aim is to develop a comprehensive design and analysis approach for various scenarios in distributed optimization and games over networks.
To achieve this unification, we begin by introducing the concept of Variational Inequality (VI), as it offers a unified view on both distributed optimization and games~\cite{scutari2010convex,scutari2014real}.
In detail, given a subset $\Xv \subseteq \R^\n$ and a vector-valued function $F: \Xv \to \R^{\n}$,
a finite-dimensional VI problem $\VI{\Xv,\F}$ consists in finding $\xstar \in \Xv$ such that
\begin{align}\label{eq:VI_problem_generic}
	(x - \xstar)\T \F(\xstar) \ge 0, \quad \forall x \in \Xv.
\end{align}
A wide range of problems arising in optimization and games over networks can be modeled as particular instances of~\eqref{eq:VI_problem_generic} referred to as \emph{multi-agent} VI problem $\VI{\Xv,\F}$.
More in detail, in the multi-agent case, given $\set := \until{N}$, the decision variable $x$ can be written as $x := \col{x_1,\dots,x_N}$ with $x_i \in \R^{\n_i}$ and $\sum_{i=1}^N \n_i = \n$, while the feasible set reads as $\Xv \!= \!\Xv_1 \! \times\! \dots \!\times\! \Xv_N$ with $\Xv_i \subseteq \R^{\n_i}$ for all $i \in \set$, and $\F$ is
\begin{align*}
	\F(x) = \begin{bmatrix}
		\F_1(x)\T
		&
		\hdots
		&
		\F_N(x)\T
	\end{bmatrix}\T,
\end{align*}
where $F_i: \R^{\n} \to \R^{\n_i}$ for all $i \in \set$.
Hence, a multi-agent VI reads as
\begin{align}\label{eq:VI_problem_distributed}
	\left(
		\begin{bmatrix}
			x_1\\
			\vdots\\
			x_N
		\end{bmatrix} - 
		\begin{bmatrix}
			\xstar_1\\
			\vdots\\
			\xstar_N
		\end{bmatrix}	
	\right)\T 
	\begin{bmatrix}
		\F_1(\xstar)
		\\
		\vdots
		\\
		\F_N(\xstar)
	\end{bmatrix} \ge 0, \quad \forall x \in \Xv,
\end{align}
where $\xstar := \col{\xstar_1,\dots,\xstar_N} \in \R^\n$ with $\xstari \in \R^{\n_i}$ for all $i \in \set$.
We denote as $\Xvstar \subseteq \Xv$ the set containing all (possible) solutions to problem~\eqref{eq:VI_problem_distributed}.
In this paper, we focus on a comprehensive development of distributed
strategies to address problem~\eqref{eq:VI_problem_distributed}. %
Specifically, we focus on networks of $N$ agents communicating
according to a graph $\cG =(\set,\cE)$, where
$\set$ is the set of agents and $\cE \subseteq \set \times \set$ is the set of edges. 
The symbol $\cN_i$ denotes the in-neighbor set of agent $i$, namely
$\mathcal{N}_i := \{j\in\set \mid (j,i)\in\cE\}$.
Finally, the in-degree of agent $i$ is the cardinality of $\cN_i$ and
we use $\degi := |\cN_i|$ to denote it.

In the next, we provide an overview of popular frameworks addressed in the context of optimization and games over networks and we cast them in the form of the VI~\eqref{eq:VI_problem_distributed}.

\subsubsection*{Distributed Consensus Optimization} $N$ agents aim at
cooperatively minimizing the sum of functions all depending on a common
variable. 
Formally, this kind of problem reads as
\begin{align}
	\min_{x \in \R^{d}} \sum_{i=1}^N \f_i(x),\label{eq:consensus_optimization}
\end{align}
where, for all $i \in \set$, $\f_i: \R^{d} \to \R$ is the cost function associated to agent $i$.
We cast problem~\eqref{eq:consensus_optimization} in the form of~\eqref{eq:VI_problem_distributed} by setting 
\begin{align*}
	\F(x) \!=\! 
	\begin{bmatrix}
		\sum_{i=1}^N\!\! \nabla\f_i(x_1)\T
		&
		\!\hdots\!
		&
		\sum_{i=1}^N \!\! \nabla\f_i(x_N)\T
	\end{bmatrix}\T\!\!, 
	\hspace{.1cm} 
	\Xv \!=\! \R^{Nd}. 
\end{align*}

\subsubsection*{Distributed Constraint-Coupled Optimization} in this setup, the coupling among the agents' variables occurs in the constraints rather than in the objective functions.
Indeed, it reads as
\begin{align} 
	\begin{split}\label{eq:constr_coupled}
		\min_{(x_1,\ldots, x_N) \in \R^{\n}} \: & \: \sum_{i=1}^N \f_i(x_i)
		\\
		\subj \: & \: \sum_{i=1}^N A_ix_i\leq \sum_{i=1}^N b_i,
	\end{split}
\end{align}
where each $f_i: \R^{\n_i} \to \R$ is the local cost function of agent
$i$ depending on the associated decision variable $x_i \in \R^{\n_i}$,
while the elements $A_i \in \R^{\m \times \n_i}$ and $b_i \in \R^{\m}$ model the feasible set.
The whole decision variable is denoted as
$x := \col{x_1,\dots,x_N} \in \R^{\n}$ with $\n := \sum_{i=1}^N
\n_i$.
This problem structure does not necessarily impose consensus among the agents' estimates.
However, as we will show also in Section~\ref{sec:overview_of_algorithms}, because of the coupling in the constraints, every single agent cannot independently solve a local problem with local information.
We note that problem~\eqref{eq:constr_coupled} can be cast in the
form~\eqref{eq:VI_problem_distributed} with
$\Xv = \big\{x \in \R^{\n} \big| \sum_{i=1}^N (A_ix_i - b_i) \leq
0\big\}$ and
\begin{align*}
	\F(x) = \begin{bmatrix} \nabla\f_1(x_1)\T
		&
		\hdots
		&
		\nabla\f_N(x_N)\T
	\end{bmatrix}\T. %
\end{align*}
We recall that we want to satisfy the constraint formalized by the feasible set $\Xv$ in an asymptotic sense.

\subsubsection*{Distributed Aggregative Optimization} $N$ agents still cooperate to minimize the sum of $N$ functions, but now they also depend on a so-called \emph{aggregative} variable.
Formally, we have
\begin{align}\label{eq:aggregative_optimization_problem}
	\begin{split}
		\min_{(x_1,\dots,x_N) \in \R^{\n}} \: & \: \sum_{i=1}^{N}\f_i(x_i,\sigma(x)),
	\end{split}
\end{align}
in which $x := \col{x_1, \dots, x_N} \in \R^\n$ is the global decision vector, with each $x_i \in \R^{\n_i}$ and $\n = \sum_{i=1}^{N} \n_i$.
Each function $\f_i: \R^{\n_i} \times \R^d \to \R$ represents the local objective function of agent $i$, while the aggregative variable $\sigma(x)$ has the form
\begin{align}\label{eq:sigma}
    \sigma(x) := \dfrac{1}{N}\sum_{i=1}^{N}\phii(x_i),
\end{align}
where $\phi_i: \R^{\n_i} \to \R^d$ gives the $i$-th contribution to $\sigma(x)$. %
We cast problem~\eqref{eq:aggregative_optimization_problem} in the form of~\eqref{eq:VI_problem_distributed} by setting $\Xv = \R^{\n}$ and
\begin{align*}
	\F(x) \!\! = \!\! \begin{bmatrix}
		\dfrac{\partial \sum_{i=1}^N \f_i(x_i,\sigma(x))}{\partial x_1}
		&
		\!\! \hdots \!\!
		&
		\dfrac{\partial \sum_{i=1}^N \f_i(x_i,\sigma(x))}{\partial x_N}
	\end{bmatrix}\T\!\!\!\!. %
\end{align*}

\subsubsection*{Aggregative Games over Networks with Linear Coupling Constraints} $N$ agents compete with each other with the aim of minimizing their own costs depending on the strategies of all the other agents.
Formally, agent $i$ aims at solving
\begin{equation}
	\begin{aligned}\label{eq:games}
		&\underset{x_i \in \R^{\n_i}}{\min} && \Ji(x_i, \sigma(x))\\
		&~\textrm{ s.t. } && \sum_{j=1}^N A_jx_j \leq \sum_{j=1}^N b_j,
	\end{aligned}
\end{equation}
where the elements $A_i \in \R^{m \times n_i}$ and $b_i \in \R^{\m}$ model the set of feasible strategies, while the cost function $\map{\Ji}{\R^{\n_i} \times \R^d}{\R}$ depends on the $i$-th individual strategy $x_i \in \R^{\n_i}$, as well as on the aggregative variable $\sigma(x) \in \R^d$ defined as in~\eqref{eq:sigma}, with $x := \col{x_1,\dots,x_N} \in \R^n$, $\n := \sum_{i=1}^N n_i$. 
We define the constraint functions $c_i(x_i) := A_ix_i - b_i$, $\cmi(\xmi) := \sum_{j \in \set \setminus \{i\}}(A_j x_j - b_j)$, and $c(x) := c_i(x_i) + \cmi(\xmi) = Ax - b$,
where $\xmi := \col{x_1,\dots,x_{i-1},x_{i+1},\dots,x_{N}} \in \R^{\n - \n_i}$, $A := \left[ A_1 \, \dots \, A_N\right] \in \R^{m\times \n}$, and $b := \sum_{i=1}^N b_i$. 
Then, the collective vector of strategies $x$ must belong to the feasible set $\cC := \{x \in \R^{\n} \mid c(x) \leq 0\} \subseteq \R^\n$. %
We refer to any equilibrium solution to the set of inter-dependent optimization problems~\eqref{eq:games} as aggregative Generalized Nash Equilibrium (GNE)~\cite{FacchineiKanzowGNE2010} (or simply GNE), which is formally defined as follows.
\begin{definition}[\textup{GNE} \cite{FacchineiKanzowGNE2010}]\label{Def:GNE}
A collective vector of strategies $\xstar \in \cC$ is a GNE of \eqref{eq:games} if, for all $i \in \set$, we have:
\begin{equation*}
	\Ji(\xstari, \sigma(\xstar)) \leq \underset{x_i\in\cCi(\xstarmi)}{\min} \, \Ji(x_i,\tfrac{1}{N}\phii(x_i) + \sigmai(\xstarmi)), %
\end{equation*}
with $\cCi(\xmi) := \{x_i\in \X_i \mid A_ix_i \leq b_i - \cmi(\xmi)\}$, $\xstarmi := \col{\xstar_1,\dots,\xstar_{i-1},\xstar_{i+1},\dots,\xstar_N}$, and $\sigmai(\xstarmi) := \sum_{j=1,j \ne i}^N \phi_j(\xstar_j)$.\oprocend
\end{definition} 
By addressing a VI in the form of~\eqref{eq:VI_problem_distributed} in which $\Xv = \cC$ and
\begin{align*}%
	\F(x) = \begin{bmatrix}
		\dfrac{\partial \J_1(x_1,\sigma(x))}{\partial x_1}
		&
		\hdots
		&
		\dfrac{\partial \J_N(x_N,\sigma(x))}{\partial x_N}
	\end{bmatrix}\T,%
\end{align*}
one finds a subset of the GNEs of problem~\eqref{eq:games} named \emph{variational} GNEs (v-GNEs).
As in the constraint-coupled setup, the constraint formalized by the feasible set $\Xv$ must be satisfied in an asymptotic sense.
\begin{remark}
	Many existing works focus on games in which each cost function $J_i$ depends on all the single strategies $x_j$, namely each agent aims at minimizing $\J_i(x_i,\xmi)$ (see, e.g.,~\cite{cenedese2021asynchronous,bianchi2022fast,belgioioso2022distributed}).
	However, we remark that the aggregative game~\eqref{eq:games} recovers this class of games by setting
		$\sigma(x) := x$,
	i.e., players need to access the whole vector of strategies.\oprocend
\end{remark}

\section{Unifying Theory for distributed algorithms: Systematic Design and Formal Guarantees}
\label{sec:systematic_design}

In this section, we provide a comprehensive approach to design an
iterative distributed algorithm aimed at solving
problem~\eqref{eq:VI_problem_distributed}.
In Section~\ref{sec:centralized_algortihm}, we start from a \centralized/ algorithm
(for optimization or games) run by $N$ agents aided by a central unit named \emph{aggregator} (or \emph{master}).
The aggregator collects/spreads data from/to all the agents of the
network and performs computations involving all variables.
Then, in Section~\ref{sec:consensus_oriented}, we develop a naive
version of a fully distributed algorithm in which we interlace the
original optimization-oriented scheme with an inner, \cnsor one.
Specifically, we introduce local proxies that asymptotically work as the aggregator.
Finally, in Section~\ref{sec:toward_a_distributed}, we suitably modify
this naive version to simultaneously run, in a distributed
fashion, both the optimization- and \cnsor schemes.

\subsection{\Centralized/ Meta-Algorithm for Optimization and Games}
\label{sec:centralized_algortihm}

We introduce a generic \centralized/ method to address problem~\eqref{eq:VI_problem_distributed}.
This scheme serves as an inspiring algorithm that we aim to replicate in a distributed manner. 
It is intentionally kept generic and does not represent a specific methodology.
Hence, we refer to it as the \emph{\centralized/ meta-algorithm}.
In detail, given the iteration index $\iter \in \N$
and $\nstatei \in \N$, each agent $i \in \set$, at iteration
$\iter$, maintains a local algorithmic state
$\stateit \in \R^{\nstatei}$ and a local output $\xit \in \R^{\n_i}$ representing the current estimate of the $i$-th block of a solution to problem~\eqref{eq:VI_problem_distributed}.
Although the variables $\stateit$ and $\xit$ are usually related, they do not necessarily coincide.
Indeed, the algorithmic state $\stateit$ may also contain, e.g., local
multipliers $\lit$ to enforce local and/or global constraints (see,
e.g.,~\cite{carnevale2022tracking}), as well as local momentum terms
$m_i\ud\iter$ to improve the convergence features of the algorithm
(see, e.g.,~\cite{carnevale2020gtadam}).
Agent $i$ iteratively updates $(\stateit,\xit)$ according to a
specific protocol implementable in a \centralized/ manner. 
Posing $\statet := \col{\statet_1,\dots,\statet_N} \in \R^{\nstate}$ with $\nstate \!:=\! \sum_{i=1}^N \nstatei$, we generally describe this centralized meta-algorithm as 
\begin{subequations}\label{eq:local_centralized_method}
	\begin{align}
		\stateitp &= \ali(\stateit,\agg(\statet))\label{eq:local_centralized_method_state}
		\\
		\xit &= \outi(\stateit),\label{eq:local_centralized_method_output}
	\end{align}
\end{subequations}
where $\ali: \R^{\nstatei} \times \R^\na \to \R^{\nstatei}$ represents the
local algorithm dynamics, $\outi: \R^{\nstatei} \to \R^{\n_i}$
represents the local output map, while $\agg: \R^{\nstate} \to \R^\na$
is an \emph{aggregation} function that couples all the local
quantities $\stateit$.
Specifically, $\agg(\state)$ represents a solution to a $\agg$-consensus
problem, i.e., the problem of finding the collective information
(e.g., the mean, the maximum, the minimum, etc...) encoded in $\agg$
about the local quantities $\state_1, \dots, \state_N$.
The specific form of $\ali$ is motivated by the distributed paradigm we aim at aligning with.
Indeed, we explicitly identified the dependency on the local (and available)
information $\stateit$ and the global (and unavailable) one encoded in $\agg(\statet)$, which needs to be computed by the aggregator.
In Section~\ref{sec:federated_algorithm}, for each setup presented in Section~\ref{sec:problem}, we provide a centralized algorithm that effectively solves the considered setup and we explicitly show that it matches the general form considered in~\eqref{eq:local_centralized_method}.
The column-stacked version of~\eqref{eq:local_centralized_method}
reads as%
\begin{subequations}\label{eq:centralized_method}
	\begin{align}
		\statetp &= \al(\statet,\1\agg(\statet))\label{eq:centralized_method_state}
		\\
		\xt &= \out(\statet),\label{eq:centralized_method_output}
	\end{align}
\end{subequations}
in which $\xt := \col{\xt_1,\dots,\xt_N} \in \R^{\n}$, while $\al: \R^{\nstate} \times \R^{N\na} \to \R^{\nstate}$ and $\out: \R^{\nstate} \to \R^{\n}$ read as 
\begin{align*}
	\al(\state,v) &:= \begin{bmatrix}
		\al_1(\state_1,v_1)\T
		&
		\hdots
		&
		\al_N(\state_N,v_N)\T
	\end{bmatrix}\T
	\\
	\out(\state) &:= \begin{bmatrix}
		\out_1(\state_1)\T
		&
		\hdots 
		&
		\out_N(\state_N)\T
	\end{bmatrix}\T,
\end{align*}
where $v := \col{v_1,\dots,v_N} \in \R^{N\na}$ with $v_i \in \R^\na$ for all $i \in \set$.

\subsection{Naive Double Loop Distributed Algorithm}
\label{sec:consensus_oriented}

In a network of peer-to-peer agents, the central aggregator is not available and, thus, the update~\eqref{eq:centralized_method_state} cannot be implemented. 
To compensate for this lack of knowledge, we equip each agent $i$ with an auxiliary variable $\zit \in \R^{\nzi}$ providing a local proxy $\taggi(\stateit,\zit) \in \R^{\na}$
about the unavailable, aggregate quantity $\agg(\state)$.
In other words, each $\zit$ is the local state variable of a consensus-oriented mechanism that we need to embed in the overall algorithm to make its execution possible in a distributed fashion.
We do that by modifying the local update~\eqref{eq:local_centralized_method_state} as
\begin{align}\label{eq:local_update_with_proxies}
	\stateitp = \ali(\stateit,\underbrace{\taggi(\stateit,\zit)}_{\mathclap{\text{proxy for } \agg(\statet)}}).
\end{align}
Next, we show a naive distributed algorithm, including an inner loop
for the update of $\zit$, that helps to understand the rationale of the
actual distributed algorithmic framework and the line of proof that we will present
in the following sections.
To this end, we introduce an inner iteration index $\initer \in \N$ and, for each $i \in \set$, we consider the double-loop algorithm
\begin{subequations}\label{eq:doublescale}
	\begin{align}
		\text{{\bf{for }}}&\iter=1, 2, \dots
		\notag\\
		&\stateitp = \ali(\stateit,\taggi(\stateit,\z_i\ud{\iter,\infty}))\label{eq:doublescale_opt}
		\\
		&\text{{\bf{for }}}\initer=1, 2, \dots
		\notag\\
		&\hspace{.6cm}\zittp = \tri(\statenit,\znitt), \label{eq:local_update_trackers}  %
	\end{align}
\end{subequations}
where $\z_i\ud{\iter,\infty} := \lim_{\tau \to \infty}\z_i\ud{\iter,\tau}$, $\tri: \R^{\sum_{j \in \cN_i} \nstatej} \times \R^{\sum_{j \in \cN_i} \nzj} \to \R^{\nzi}$ is a generic description of the \cnsor dynamics, while %
$\statenit \in \R^{\sum_{j \in \cN_i} \nstatej}$ and $\znitt \in
\R^{\sum_{j \in \cN_i} \nzj}$ collect the quantities maintained by the in-neighbors
of agent $i$, namely 
\begin{align*}
	\statenit &:= \col{\statejt}_{j \in \cN_i}, \quad \znitt := \col{\zjtt}_{j \in \cN_i},
\end{align*}
for all $i \in \set$.
In Section~\ref{sec:dynamic_average_consensus}, we will provide explicit examples of this consensus-oriented dynamics in the case in which the $\agg$-consensus problem reconstruction coincides with an average consensus problem, i.e., in the case in which $\agg(\state)$ represents the mean of local quantities (e.g., gradients, solution estimates, or constraints).
To generalize, as it will be clear in the sequel, the inner update must be designed to 
asymptotically solve the $\agg$-consensus problem, i.e., the trajectories of~\eqref{eq:local_update_trackers} must satisfy
\begin{align*}
	\lim_{\initer\to\infty}\taggi(\stateit,\zitt) = \agg(\statet),%
\end{align*}
for all $\statet \in \R^{\nstate}$ and $i \in \set$.
We conclude this part by providing the column-stacked version of~\eqref{eq:local_update_trackers}, namely%
\begin{align}\label{eq:consensus_oriented} 
	\zttp = \tr(\statet,\ztt),
\end{align}
where, given $\nz := \sum_{i=i}^N \nzi$, we introduced $\tr(\state,\z):= \begin{bmatrix}
		\tr_1(\state_{\cN_1},\z_{\cN_1})\T
		&
		\hdots
		&
		\tr_N(\state_{\cN_N},\z_{\cN_N})\T
	\end{bmatrix}\T \in \R^{\nz}$.
We point out that algorithm~\eqref{eq:doublescale} is not practically implementable since agents should run the inner loop forever or share a common iteration index to switch between the inner and outer loops.

\subsection{Distributed Meta-Algorithm for Optimization and Games}
\label{sec:toward_a_distributed}

We follow a key intuition based on singular perturbations to
transform~\eqref{eq:doublescale} in a fully distributed
algorithm without an inner loop.
To do that, we suitably modify~\eqref{eq:doublescale_opt} and~\eqref{eq:local_update_trackers} to create an interconnected system with
a fast and a slow dynamics.
To this end, we introduce a parameter $\pr > 0$ allowing for an
arbitrary tuning of the variations of $\statet$, so
that~\eqref{eq:local_update_with_proxies} becomes
\begin{align}
	\stateitp = \stateit + \pr(\ali(\stateit,\taggi(\stateit,\zit)) -
  \stateit). \label{eq:local_update_with_step}
\end{align}
By simultaneously running~\eqref{eq:local_update_with_step} and a single step
of~\eqref{eq:local_update_trackers}, we obtain our distributed meta-algorithm for optimization and games and we report it in Algorithm~\ref{alg:distributed_scheme}.
\begin{algorithm}
\begin{algorithmic}
\For{$\iter=1, 2, \dots$}
\begin{subequations}\label{eq:local_update}
	\begin{align}
		\stateitp &= \stateit + \pr(\ali(\stateit,\taggi(\stateit,\zit)) - \stateit)\label{eq:local_update_wi}
		\\
		\zitp &= \tri(\statenit,\znit)\label{eq:local_update_zi}
		\\
		\xit &= \outi(\stateit)
	\end{align}
\end{subequations}
\EndFor
\end{algorithmic}
\caption{Distributed Meta-Algorithm (Agent $i$)}
\label{alg:distributed_scheme}
\end{algorithm}
We remark that it only uses variables locally available for agent $i$ and, thus, can be implemented in a fully distributed manner.
Although our primary objective is to provide a tool for developing and analyzing novel distributed methods, it is noteworthy that the general form of the generic distributed meta-algorithm reported in Algorithm~\ref{alg:distributed_scheme} encompasses several existing methods.
Just to name a few, we mention ADMM-Tracking Gradient~\cite{carnevale2023admm} for consensus optimization~\eqref{eq:consensus_optimization}, Distributed Aggregative Gradient Tracking~\cite{li2021distributed} for aggregative optimization~\eqref{eq:aggregative_optimization_problem}, and Primal TRADES~\cite{carnevale2022tracking} for aggregative games with local constraints~\eqref{eq:games}.

By collecting all the local updates~\eqref{eq:local_update}, the stacked-column description of our distributed meta-algorithm reads as
\begin{subequations}\label{eq:global_update}
	\begin{align}
		\statetp &= \statet + \pr(\al(\statet,\tagg(\statet,\zt))-\statet)\label{eq:global_update_chi}
		\\
		\ztp &= \tr(\statet,\zt)\label{eq:global_update_z}
		\\
		\xt &= \out(\statet),
	\end{align}
\end{subequations}
where $\tagg(\state,\z) := \begin{bmatrix}
		\tagg_1(\state_1,\z_1)\T
		&
		\hdots
		&
		\tagg_N(\state_N,\z_N)\T
	\end{bmatrix}\T$.
The obtained distributed method interconnects the optimization-oriented subsystem~\eqref{eq:global_update_chi} and the \cnsor one~\eqref{eq:global_update_z}, see Fig.~\ref{fig:scheme} for a schematic view.
\begin{figure}
	\centering
	\includegraphics[scale=1.1]{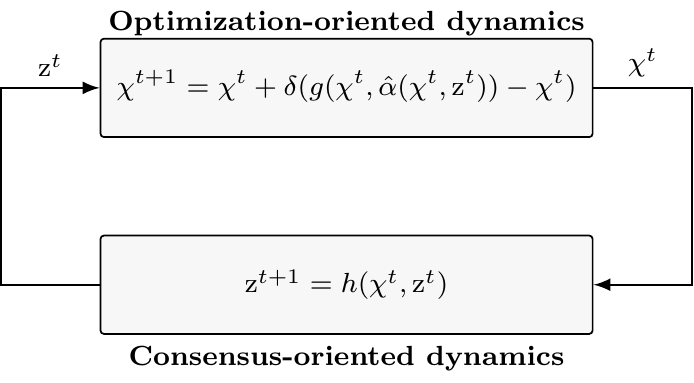}
	\caption{Block diagram of~\eqref{eq:global_update}.}
	\label{fig:scheme}
\end{figure}
Intuitively, with a sufficiently small $\pr$, the \cnsor scheme deals with a ``quasi-static'' $\statet$,
so that $\tagg(\statet,\zt) \approx \1\agg(\statet)$.
Then, the optimization-oriented scheme mimics a low-pass filter of the
desired and centralized meta-algorithm~\eqref{eq:centralized_method}.
Based on this intuition, we formalize as follows the necessary
properties of the \centralized/ optimization meta-algorithm~\eqref{eq:centralized_method}
and the \cnsor scheme~\eqref{eq:consensus_oriented} separately considered.
We start from the required properties about the centralized method~\eqref{eq:centralized_method}.
Then, in Section~\ref{sec:federated_algorithm}, for each setup considered in Section~\ref{sec:problem}, we provide a centralized method that satisfies these assumptions.
\begin{assumption}[Centralized method~\eqref{eq:centralized_method} equilibria]\label{ass:centralized_equilibria}
$\Xstar$ is non-empty and there exist non-empty sets $\cXs,\cSc \! \subseteq \! \R^{\nstate}$ such that
\begin{enumerate}
	\item $\cXs \! := \! \{\statestar \in \R^{\nstate} \mid \statestar = \al(\statestar,\1\agg(\statestar)), \out(\statestar) \in \Xstar\}$.
	\\
	\item The set $\cSc \supset \cXs$ is convex and, for all $(\state,v) \in \cSc \times \R^{N\na}$, it holds $\al(\state,v) \in \cSc$.
	\oprocend
\end{enumerate}
\end{assumption}
Assumption~\ref{ass:centralized_equilibria} introduces the set $\cXs$ containing the equilibria $\statestar$ of the centralized method~\eqref{eq:centralized_method}
whose output $\xstar = \out(\statestar)$ is a solution to the variational inequality~\eqref{eq:VI_problem_distributed}.
As for the set $\cSc$, we clarify its role as follows.
\begin{remark}\label{rem:init_chi}
	  We introduced the forward-invariant set $\cSc$ to include optimization-oriented methods requiring a proper initialization (e.g., projection-based algorithms) and, thus, to enhance the class of setups and methods that can be captured by our framework.
	  To further stress that this is not a limitation, we note that in the case of initialization-free methods (e.g., in unconstrained programs), since we do not require the boundedness of $\cSc$, we can simply set $\cSc \equiv \R^{\nstate}$.\oprocend
\end{remark}
The next assumption characterizes the convergence properties of the centralized method~\eqref{eq:centralized_method} by introducing a suitable descent-like function $\Vc$.
Such a function $\Vc$ may be, e.g., the overall cost function in the case of unconstrained programs (see~\cite{carnevale2022nonconvex}) or the so-called gap function (see~\cite{mastroeni2003gap}) in the case of games.
\begin{assumption}[Convergence of centralized method~\eqref{eq:centralized_method}]\label{ass:centralized_convergence}
	There exist a radially unbounded, continuously differentiable $\Vc: \R^{\nstate} \to \R$, $\dd: \R^{\nstate} \to \R_+$, and $c_1, \cf, \cff, \cfff > 0$
	such that%
	\begin{subequations}\label{eq:Vc}
		\begin{align}
			\nabla\Vc (\state) \T  \left(\al(\state,\1\agg(\state))  -  \state\right)  + \tfrac{c_3}{2}\norm{\al(\state,\1\agg(\state)
			) -   \state}^2
			&\leq  
			-c_1 \dd(\state)^2
			\label{eq:minus_Vc}
			\\
			\norm{\nabla\Vc(\state)} &\leq\cf\dd(\state)
			\label{eq:nabla_vc_dd}			
			\\
			\norm{\nabla\Vc(\state) - \nabla\Vc(\state^\prime)}  &\leq  \cff\norm{\state - \state^\prime}\label{eq:lipschitz_Vc}
			\\
			\norm{\al(\chi,\1\agg(\chi)) - \state}
			&\leq 
			\cfff\dd(\state),
			\label{eq:bound_dd}
		\end{align}
	\end{subequations}
	for all $\state, \state^\prime \in \cSc$. %
	Moreover, 
	$\statestar \in \{\state \in \R^{\nstate} \mid \dd(\state) = 0\} \implies \statestar \in \cXs$.
	Further, $\al$ and $\out$ are $\lippa$- and $\lipp_{\out}$-Lipschitz continuous over $\cSc \times \R^{N\na}$ and
	$\cSc$, for some $\lippa, \lipp_{\out} \!>\! 0$.
	\oprocend
\end{assumption}
The next assumption focuses on variational inequalities~\eqref{eq:VI_problem_distributed} with a unique solution $\xstar \in \Xvstar$ and centralized methods~\eqref{eq:centralized_method} with exponential stability guarantees or, equivalently, a linear convergence rate.
These stronger guarantees are commonly achieved in the case of strongly convex cost functions (optimization) or strongly monotone pseudo-gradients (game theory).
\begin{assumption}[Exp. stability for centralized method~\eqref{eq:centralized_method}]
	\label{ass:additional_assumption}
	Given $\statestar \in \R^{\nstate}$ such that $\xstar = \out(\statestar) \in \Xvstar$, it holds
	\begin{subequations}\label{eq:additional_condtions}
		\begin{align} 
			\underbar{c}\norm{\state - \statestar}^2 \leq \Vc(\state) &\leq \bar{c}\norm{\state - \statestar}^2\label{eq:additional_condtions_bounds}
			\\
			\dd(\state)^2 &\ge c\norm{\state - \statestar}^2,
			\label{eq:additional_condtions_minus}
		\end{align}
	\end{subequations}
	for all $\state \in \cSc$ and some $\underbar{c}, \bar{c}, c > 0$.\oprocend
\end{assumption}
\begin{remark}
	The role of Assumptions~\ref{ass:centralized_convergence} (resp.~\ref{ass:additional_assumption}) is to frame convergence results of given centralized methods within a system theory interpretation based on LaSalle's Invariance Principle (resp. Lyapunov's stability theory).
	Indeed, we recall that our goal is to develop a systematic approach to designing and analyzing distributed versions of existing optimization-oriented methods.
	Hence, we require convergence results of these inspiring methods to be available.\oprocend
\end{remark}
Analogously, we will now formalize the necessary features of the \cnsor scheme~\eqref{eq:consensus_oriented}.
Section~\ref{sec:dynamic_average_consensus} provides an overview of existing schemes satisfying these assumptions.

It is possible to decompose the most popular \cnsor dynamics (see the tutorial~\cite{kia2019tutorial} about dynamic average consensus) into two suitable components, one of which has a forward-invariant set, while the other has a stable equilibrium parametrized in $\state$.
In turn, $\agg$-consensus is achieved in configurations where the first part lies on the mentioned set and the other one in this equilibrium.
To make these schemes eligible in our systematic procedure, we introduce a change of variables $\TT: \R^{\nz} \to \R^{\nz}$ that transforms $\z$ according to
\begin{align}
	\z \longmapsto
		\begin{bmatrix}
			\bz
			\\
			\pz 
		\end{bmatrix}
		:= 
		\TT(\z) := 
		\begin{bmatrix}
			\Tb(\z)
			\\
			\Tp(\z)
		\end{bmatrix},\label{eq:change_of_variables}
\end{align}
with $\Tb: \R^{\nz} \to \R^{\nb}$ and $\Tp: \R^{\nz} \to \R^{\np}$ with $\nb + \np = \nz$.
The transformation $\TT$ allows us to rewrite system~\eqref{eq:global_update} as 
\begin{subequations}\label{eq:global_update_decomposed_int}
	\begin{align}
		\statetp &= \statet + \delta(\al(\statet,\tagg(\statet,\TTi(\col{\bzt,\pzt}))) - \statet)
		\label{eq:global_update_decomposed_int_chi}
		\\
		\bztp &= \Tb(\tr(\statet,\TTi(\col{\bzt,\pzt})))
		\label{eq:global_update_decomposed_int_bz}
		\\
		\pztp &= \Tp(\tr(\statet,\TTi(\col{\bzt,\pzt})))
		\label{eq:global_update_decomposed_int_pz}
		\\
		\xt &= \out(\statet).
	\end{align}
\end{subequations}
In these new coordinates, we suppose the following invariance and orthogonality conditions.
We note that these conditions are common to the most widely used \cnsor schemes and, among others, to those presented later in Section~\ref{sec:dynamic_average_consensus}, which serve as building blocks for our systematic methodology.
\begin{assumption}(Invariance and orthogonality)
	\label{ass:orthogonality}
	There exists a forward-invariant set $\cSbz \subseteq \R^{\nb}$ for~\eqref{eq:global_update_decomposed_int_bz}.
	Moreover, there exist $\ptagg: \R^{\nstate} \! \times \! \R^{\np} \! \to \! \R^{N\na}$ and $\ptr: \R^{\nstate} \! \times \! \R^{\np} \! \to \! \R^{\np}$ such that%
	\begin{subequations}\label{eq:orthogonality}
		\begin{align}
			\tagg(\state,\TTi(\col{\bz,\pz})) &= \ptagg(\state,\pz)\label{eq:ptagg}
			\\
			\Tp(\tr(\state,\TTi(\col{\bz,\pz}))) &= \ptr(\state,\pz),\label{eq:ptr}
		\end{align}	
	\end{subequations}	
	for all $\state \!\in\! \R^{\nstate}$, $\bz \!\in\! \cSbz$, and $\pz \!\in\! \R^{\np}$.
	Further, $\ptagg$ and $\ptr$ are Lipschitz continuous with constants $\lipp_{\tagg}, \lipp_{\tr} \!>\! 0$, respectively.\oprocend
\end{assumption}
Now, we guarantee the existence of an equilibrium of~\eqref{eq:global_update_decomposed_int_pz} parametrized in $\state$ in which the $\agg$-consensus is solved.
\begin{assumption}(Parametrized equilibrium)
	\label{ass:equilibria_consensus}
	There exists $\pzeq: \R^{\nstate} \to \R^{\np}$ such that, for all $\state \in \R^{\nstate}$, it holds
	\begin{subequations}
		\begin{align}
			\ptagg(\state,\pzeq(\state)) &= \1\agg(\state)
			\label{eq:zeq_reconstruction}
			\\
			\pzeq(\state) &= \ptr(\state,\pzeq(\state)).\label{eq:zeq_equilibrium}
		\end{align}
	\end{subequations}
	Further, $\pzeq$ is Lipschitz continuous with constant $\lippeq > 0$.\oprocend
\end{assumption}
Conditions~\eqref{eq:ptagg} and~\eqref{eq:zeq_reconstruction} ensure that if $\bz \in \cSbz$ and $\pz = \pzeq(\state)$, the $\agg$-consensus problem is solved.
Since $\cSbz$ is forward-invariant for~\eqref{eq:global_update_decomposed_int_bz}, the set $\cSz:=\{\z\in \R^{\nz} \mid\Tb(\z) \in \cSbz\}$ turns out to be forward-invariant for the whole \cnsor part~\eqref{eq:global_update_z}.
Hence, by initializing system~\eqref{eq:global_update} with $\z\ud0\in\cSz$, one ensures $\bzt \in \cSbz$ for all $\iter \in \N$.
Thus, to assess the attractiveness of the entire $\agg$-consensus locus $\{(\bz,\pz) \in \R^{\nb} \times \R^{\np} \mid \bz \in \cSbz,\pz = \pzeq(\state)\}$, we now impose asymptotic stability properties for the equilibrium $\pzeq(\chi)$.
\begin{assumption}[Stability of \cnsor dynamics]
	\label{ass:tracking}
	There exist $\Vz: \R^{\np} \to \R$ and $b_1, b_2, b_3, b_4 > 0$ such that%
	\begin{subequations}\label{eq:Vz_theorem_distributed}
		\begin{align}
			&b_1\norm{\tz}^2 \leq \Vz(\tz) \leq b_2\norm{\tz}^2\label{eq:Vz_quadratic_distributed}
			\\
			&\Vz(\ptr(\state,\tz \! + \! \pzeq(\state)) \! - \! \pzeq(\state)) \! - \! \Vz(\tz) \! \leq \! - b_3\norm{\tz}^2 
			\\
			&|\Vz(\tz) - \Vz(\tz^\prime)| \leq b_4\norm{\tz - \tz^\prime}(\norm{\tz} + \norm{\tz^\prime}),
		\end{align}
	\end{subequations}
	for all $\state \in \R^{\nstate}$ and $\tz, \tz^\prime \in \R^{\np}$.
	\oprocend
\end{assumption}
Assumption~\ref{ass:tracking} ensures that if $\bz \in \cSbz$ and $\state$ were fixed, the point $\pzeq(\state)$ would be a globally exponentially stable equilibrium uniformly in $\state \in \R^{\nstate}$ of dynamics~\eqref{eq:global_update_decomposed_int_pz}.
The following remark provides insights on $\cSz$ and the transformation~\eqref{eq:change_of_variables}.
\begin{remark}\label{rem:init_z}
	The purpose of $\cSz$ is to enlarge the class of schemes usable in our design procedure.
	Indeed, as we will see also later, some popular \cnsor schemes (see~\cite{kia2019tutorial}) widely adopted in distributed optimization (see, e.g., \cite{shi2015extra,nedic2017achieving,xi2017add}) require initialization in a subset $\cSz$.
	We remark that this choice does not represent a limitation. 
	Indeed, in the case of initialization-free protocols, since no boundedness of $\cSbz$ has been required, one may simply set $\cSbz \equiv \R^{\nb} \implies \cSz \equiv \R^{\nz}$.
	Analogously, we also note that requiring stability properties only on a portion of the dynamics of $\z$ (see Assumption~\ref{ass:tracking}) is not restrictive: for schemes exhibiting these stability properties on their whole dynamics, one can set $\nb = 0$ and $\np = \nz$.
	\oprocend
\end{remark}
Once these assumptions have been enforced, we are ready to provide the following theorem showing that the state of the distributed method~\eqref{eq:global_update} converges to a proper set associated to the solutions of the variational inequality~\eqref{eq:VI_problem_distributed}.
Moreover, under the additional Assumption~\ref{ass:additional_assumption}, we also ensure exponential stability properties
and, in turn, linear convergence of $\xt$ toward the unique $\xstar \!\in\! \Xvstar$ solving the variational inequality~\eqref{eq:VI_problem_distributed}.
\begin{theorem}\label{th:main}
	Consider~\eqref{eq:global_update} and let Assumptions~\ref{ass:centralized_equilibria}, \ref{ass:centralized_convergence},~\ref{ass:orthogonality},~\ref{ass:equilibria_consensus}, and~\ref{ass:tracking} hold for a suitable $\TT$ as in~\eqref{eq:change_of_variables}.
	Then, there exists $\bar{\pr} \in (0,1)$, such that, for all $\pr \in (0,\bar{\pr})$ and $(\state\ud 0,\z\ud 0) \in \cSc \times \cSz$, the trajectories of~\eqref{eq:global_update} satisfy%
	\begin{subequations}\label{eq:lasalle_convergence}
		\begin{align}
			\lim_{\iter \to\infty}\norm{\statet}_{\cXs} &= 0\label{eq:lasalle_convergence_chi}
			\\
			\lim_{\iter \to\infty}\norm{\Tp(\zt) - \pzeq(\statet)} &= 0\label{eq:lasalle_convergence_z}
			\\
			\lim_{\iter \to\infty}\norm{\xt}_{\Xstar} &= 0.\label{eq:lasalle_convergence_xt}
		\end{align}	
	\end{subequations}
	If further Assumption~\ref{ass:additional_assumption} holds true, then there exist $r_1, r_2, r_3 > 0$ such that, for all $\pr \in (0,\bar{\pr}$) and $(\state\ud 0,\z\ud 0) \in \cSc \times \cSz$, the trajectories of~\eqref{eq:global_update} satisfy%
	\begin{subequations}\label{eq:additional}
		\begin{align}
			\norm{
				\!
				\begin{bmatrix}
				\statet
				\\
				\Tp(\zt)
			\end{bmatrix} 
			\!\! - \!\!
			\begin{bmatrix}
				\statestar
				\\
				\pzeq(\statet)
			\end{bmatrix}} 
			\!&\leq\! 
			r_2\!\norm{\begin{bmatrix}
				\state\ud0
				\\
				\Tp(\z\ud0)
			\end{bmatrix} 
			\!\! - \!\!
			\begin{bmatrix}
				\statestar
				\\
				\pzeq(\state\ud0)
			\end{bmatrix}} \! 
			\exp(-r_1\iter)
			\label{eq:additional_result_states}
			\\
			\norm{\xt - \xstar} &\leq r_3\exp(-r_1\iter),
			\label{eq:additional_result}
		\end{align}
	\end{subequations}
	where $\xstar = \out(\statestar)$. 
\end{theorem}
\begin{proof}
	First of all, since $\cSc$ and $\cSz$ are forward-invariant for~\eqref{eq:centralized_method} and~\eqref{eq:global_update_z} (cf. Assumptions~\ref{ass:centralized_equilibria} and~\ref{ass:orthogonality}), $\cSc$ is convex (cf. Assumption~\ref{ass:centralized_equilibria}), $\pr \in (0,1)$, and $(\state\ud0,\z\ud0) \in \cSc \times \cSz$, we note that $\cSc \times \cSz$ is forward invariant for system~\eqref{eq:global_update}.
	Now, we use the change of variables~\eqref{eq:change_of_variables} and, thus, the equivalent formulation provided in~\eqref{eq:global_update_decomposed_int}.
	Since $\z\ud0 \in \cSz$, it holds $\bz\ud0 \in \cSbz$ by construction.
	Being $\cSbz$ forward invariant for~\eqref{eq:global_update_decomposed_int_bz} (cf. Assumption~\ref{ass:orthogonality}), it holds $\bzt \in \cSbz$ for all $\iter \in \N$.
	Thus, in light of the orthogonality condition~\eqref{eq:orthogonality} in Assumption~\ref{ass:orthogonality}, $\bzt$ does not affect the other system states.
	Thus, we completely ignore it and use~\eqref{eq:orthogonality} to rewrite~\eqref{eq:global_update_decomposed_int} as
	\begin{subequations}\label{eq:global_update_decomposed}
		\begin{align}
			\statetp &= \statet + \delta(\al(\statet,\ptagg(\statet,\pzt)) - \statet)
			\label{eq:global_update_decomposed_slow}
			\\
			\pztp &= \ptr(\statet,\pzt).\label{eq:global_update_decomposed_fast}
		\end{align}
	\end{subequations}
	After this preliminary phase, we analyze~\eqref{eq:global_update_decomposed} by applying Theorem~\ref{th:generic} (in Appendix~\ref{sec:SP}), i.e., we interpret it as an SP system, see their generic formulation provided in~\eqref{eq:interconnected_system}.
	In particular, $\statet$ and~\eqref{eq:global_update_decomposed_slow} are the slow state and subsystem, while $\pzt$ and~\eqref{eq:global_update_decomposed_fast} are the fast ones.
	We then proceed by checking that all the assumptions required by Theorem~\ref{th:generic} are fulfilled.
	In light of~\eqref{eq:zeq_equilibrium}, we note that $\pzeq$ provides the equilibrium function of the fast part~\eqref{eq:global_update_decomposed_fast} parametrized in $\statet$.
	Then, the so-called boundary layer system associated to~\eqref{eq:global_update_decomposed}, i.e., the fast dynamics~\eqref{eq:global_update_decomposed_fast} written using the error coordinates $\tzt := \pzt - \pzeq(\statet)$ and an arbitrarily fixed $\statet = \state \in \R^{\nstate}$, reads as 
	\begin{align}
		\tztp = \ptr(\state,\tzt + \pzeq(\state)) - \pzeq(\state).\label{eq:bl}
	\end{align}
	The Lyapunov function $\Vz$ (cf. Assumption~\ref{ass:tracking}) ensures that the origin is globally exponentially stable for~\eqref{eq:bl} as required by Theorem~\ref{th:generic} in~\eqref{eq:Vz}.
	We proceed by studying the so-called reduced system associated to~\eqref{eq:global_update_decomposed}, i.e., the slow dynamics~\eqref{eq:global_update_decomposed_slow} with $\pzt \!=\! \pzeq(\statet)$ for all $\iter \in \N$.
	This auxiliary system reads as 
	\begin{align}\label{eq:reduced_system_implicit}
		\statetp = \statet + \pr(\al(\statet,\ptagg(\statet,\pzeq(\statet)))- \statet).
	\end{align}
	The reconstruction property~\eqref{eq:zeq_reconstruction} ensures $\ptagg(\statet,\pzeq(\statet)) = \1\agg(\statet)$.
	Hence, we introduce $\ala(\state) := \al(\state,\1\agg(\state))$ to lighten up the notation and equivalently rewrite system~\eqref{eq:reduced_system_implicit} as
	\begin{align}\label{eq:reduced_system}
		\statetp = \statet + \pr(\ala(\statet) - \statet).
	\end{align}
	To apply Theorem~\ref{th:generic}, we need a function satisfying the LaSalle conditions encoded in~\eqref{eq:Vc_theorem}.
	To this end, we pick the function $\Vc$ introduced in Assumption~\ref{ass:centralized_convergence}. 
	Thus, in light of~\eqref{eq:nabla_vc_dd},\eqref{eq:lipschitz_Vc}, and~\eqref{eq:bound_dd}, the regularity conditions~\eqref{eq:nabla_vc_d},~\eqref{eq:lipschitz_vc_app}, and~\eqref{eq:slow_dd} are guaranteed.
	As for the remaining condition~\eqref{eq:Vc_minus}, we evaluate $\Delta\Vc(\statet) := \Vc(\statetp) - \Vc(\statet)$ along the trajectories of the reduced system~\eqref{eq:reduced_system}, thus obtaining 
	\begin{align}
		\Delta\Vc(\statet) 
		&= \Vc(\statet + \pr(\ala(\statet) - \statet))
		- \Vc(\statet)
		\notag\\
		&\stackrel{(a)}{\leq}
		\Vc(\statet) + \pr\nabla\Vc(\statet)\T(\ala(\statet) - \statet) 
		% \notag\\
		% &\hspace{.4cm}
		+ \pr^2\tfrac{\cff}{2}\norm{\ala(\statet) -\statet}^2 - \Vc(\statet)
		\notag\\
		&\stackrel{(b)}{\leq}
		% -\pr\big(c_1 -\pr\tfrac{\lippvc}{2}\big)\norm{\ala(\statet) -\statet}^2
		-\pr c_1\dd(\statet)^2,
	\label{eq:DeltaW_1}
	\end{align}
	where in $(a)$, in light of the Lipschitz continuity of $\nabla\Vc$ (cf.~\eqref{eq:lipschitz_Vc} in Assumption~\ref{ass:centralized_convergence}), we bound $\Vc(\statet + \pr(\ala(\statet) - \statet))$ via Descent Lemma~\cite[Prop.~6.1.2]{bertsekas2015convex}, while in $(b)$ we use~\eqref{eq:minus_Vc} and the fact that $\pr^2 \leq \pr$.
	Inequality~\eqref{eq:DeltaW_1} proves that $\Vc$ satisfies~\eqref{eq:Vc_minus} too.
	Further, Assumptions~\ref{ass:centralized_convergence} and~\ref{ass:tracking} guarantee that the Lipschitz continuity of the dynamics of~\eqref{eq:global_update_decomposed} and the function $\pzeq$ and the invariance conditions asked in~\eqref{eq:invariance} are achieved into $\cSc$ and $\R^{\np}$.
	Thus, we are in the position to apply Theorem~\ref{th:generic}.
	Namely, we use $\Vz$ and $\Vc$ to define $V: \R^{\nstate} \times \R^{\nz} \to \R$ as in~\eqref{eq:V} and
	Theorem~\ref{th:generic} guarantees there exist $\bar{\pr} \in (0,1)$ and $\psi > 0$ such that, by evaluating $\Delta V(\statet,\pzt - \pzeq(\statet)) := V(\statetp,\pztp - \pzeq(\statetp)) - V(\statet,\pzt - \pzeq(\statet))$ along the trajectories of system~\eqref{eq:global_update_decomposed}, for all $\pr \in (0,\bar{\pr})$, it holds
	\begin{align}
		\Delta V(\statet,\pzt  -  \pzeq(\statet))  \leq  
		 -  \psi\big(\dd(\statet)^2  +   \norm{\pzt  -  \pzeq(\statet)}^2\big),
		\label{eq:DeltaV_main_theorem}
	\end{align}
	for all $\iter \in \N$.
	The inequality~\eqref{eq:DeltaV_main_theorem} ensures $\Delta V(\statet,\pzt - \pzeq(\statet)) \leq 0$ for all $\iter\in\N$.
	In detail, the right-hand side of~\eqref{eq:DeltaV_main_theorem} is null when $(\statet,\pzt) \in E \subseteq \cSc \times \R^{\np}$, in which
	\begin{align}\label{eq:E_theorem}
		% E :=  \{(\state,\tz) \in\R^{\nstate} \times \R^{\nz} \mid \ala(\state) = \state,\tz  =  0\}.
		E  :=   \{(\state,\pz)  \in \cSc  \times \R^{\np}  \mid  \dd(\state)  =  0,\pz   =   \pzeq(\state)\}.
	\end{align}
	Thus, by LaSalle's invariance principle~\cite[Theorem~3.7]{ge2011invariance}, for all $\pr \in (0,\bar{\pr})$, the trajectories of system~\eqref{eq:global_update_decomposed} satisfy %
	\begin{align}\label{eq:limit}
		\lim_{\iter \to\infty}\norm{\col{\statet,
				\pzt}-\xi}_{\cM} = 0,
	\end{align}
	for all $(\state\ud0,\pz\ud0) \in \cSc \times \R^{\np}$, where $\cM \subseteq E$ denotes the largest invariant set for system~\eqref{eq:global_update_decomposed} contained in $E$ (cf.~\eqref{eq:E_theorem}). 
	Since $\statestar \!\in\! \{\state \!\in\! \R^{\nstate} \!\mid\! \dd(\state) \!=\! 0\} \!\implies\! \statestar \!\in\! \cXs$ (cf. Assumption~\ref{ass:centralized_convergence}), we have $\cM \subseteq\{(\state,\pz)\in \R^{\nstate} \times \R^{\np} \mid \state \in \cXs, \pz \!=\! \pzeq(\statestar)\}$.
	In turn, $\statestar \in \cXs \implies \out(\statestar) \in \Xvstar$ by construction (cf. Assumption~\ref{ass:centralized_equilibria}) and, thus, also~\eqref{eq:lasalle_convergence} is achieved.
	As regards the additional results~\eqref{eq:additional}, we use~\eqref{eq:additional_condtions_minus} to further bound~\eqref{eq:DeltaV_main_theorem} as 
	\begin{align}
		\Delta V(\statet,\pzt  -  \pzeq(\statet)) \leq -\psi\big(c\norm{\statet  -  \statestar}^2  +  \norm{\pzt  -  \pzeq(\statet)}^2\big).\label{eq:DeltaV_additional}
	\end{align}
	By using the quadratic bounds~\eqref{eq:additional_condtions_bounds} and~\eqref{eq:Vz_quadratic_distributed}, we get 
	\begin{align}
		\tilde{\underbar{c}}\norm{\begin{bmatrix}\state  - \statestar\\\pz  -  \pzeq(\state)\end{bmatrix}}^2   \leq  V(\state,\pz  -  \pzeq(\state))  \leq  \tilde{\bar{c}}\norm{\begin{bmatrix}\state-\statestar\\\pz  -  \pzeq(\state)\end{bmatrix}}^2,\notag%
	\end{align}
	for all $(\state\ud0,\tz\ud0) \in \cSc\times\R^{\np}$, where $\tilde{\underbar{c}}:=\min\{\underbar{c}, b_1\}$ and $\tilde{\bar{c}}:=\max\{\bar{c},b_2\}$.
	In light of these quadratic bounds and inequality~\eqref{eq:DeltaV_additional}, we invoke~\cite[Th.~13.2]{chellaboina2008nonlinear} to claim that $(\statestar,\pzeq(\statestar))$ is exponentially stable for~\eqref{eq:global_update_decomposed} and, thus, result~\eqref{eq:additional_result_states} is achieved (see the proof of~\cite[Th.~13.2]{chellaboina2008nonlinear} to quantify the constants $r_1$ and $r_2$ in~\eqref{eq:additional_result_states}).
	The proof of~\eqref{eq:additional_result} follows by recalling that $\xstar = \out(\statestar)$ and setting $r_3 \!=\! \lipp_{\out}r_2\norm{\col{\state\ud0 \!-\! \statestar,\pz\ud0\!-\!\pzeq(\state\ud0)}}$.
\end{proof}
We note that the bound $\bar{\pr}$ on the parameter $\pr$ in Theorem~\ref{th:main} is uniform in $(\state\ud0,\z\ud0) \in \cSc \times \cSz$ (see Remarks~\ref{rem:init_chi} and~\ref{rem:init_z} to recall the purpose of these sets).
Hence, using blocks having $\cSc \equiv \R^{\nstate}$ and $\cSz \equiv \R^{\nz}$, Theorem~\ref{th:main} provides global results.

\section{Building Blocks for\\ Distributed Optimization and Games Setups}
\label{sec:overview_of_algorithms}

In this section, we provide implementations of the meta-algorithm~\eqref{eq:centralized_method} and \cnsor protocols~\eqref{eq:consensus_oriented} for the scenarios considered in Section~\ref{sec:problem}.
This section, combined with the systematic design and analysis offered in Section~\ref{sec:systematic_design}, paves the way for designing a large variety of distributed algorithms in the form~\eqref{eq:global_update} for optimization and games.

First, for each setup introduced in Section~\ref{sec:problem}, we provide a \centralized/ method matching Assumptions~\ref{ass:centralized_equilibria} and~\ref{ass:centralized_convergence} (and, some of them, the additional Assumption~\ref{ass:additional_assumption}).
The considered \centralized/ methods involve aggregation functions that can be computed via dynamic average consensus.
Thus, we then provide a set of dynamic average consensus schemes that satisfy Assumptions~\ref{ass:orthogonality},~\ref{ass:equilibria_consensus}, and~\ref{ass:tracking}.

\subsection{\Centralized/ Methods for Optimization and Games}
\label{sec:federated_algorithm}

For each setup considered in Section~\ref{sec:problem}, we now provide a \centralized/ method satisfying Assumptions~\ref{ass:centralized_equilibria} and~\ref{ass:centralized_convergence}.

\subsubsection*{Consensus Optimization} we now present a \centralized/ method for problems in the form~\eqref{eq:consensus_optimization}.

First, we introduce an augmented cost function $\tfc: \R^{Nd} \to \R$ defined as 
\begin{align}
	\tfc(x) \!:=\! \frac{\const}{2} x\T \left(I - \frac{\1\1\T}{N}\right)x + N \sum_{i=1}^N \f_i\bigg(\dfrac{1}{N}\sum_{j=1}^N x_j\bigg),\label{eq:tilde_f}
\end{align}
where $x := \col{x_1,\dots,x_N} \in \R^{Nd}$ and $\const > 0$ can be arbitrarily chosen.
Such an augmented cost allows us to equivalently rewrite problem~\eqref{eq:consensus_optimization} as
\begin{align}
	\min_{(x_1, \dots, x_N) \in \R^{Nd}} \tfc(x),
	\label{eq:consensus_optimization_alternative}
\end{align}
namely, $\xstar \in \R^{d}$ is a stationary point of problem~\eqref{eq:consensus_optimization} if and only if $\1\xstar \in \R^{Nd}$ is a stationary point of~\eqref{eq:consensus_optimization_alternative}.
Hence, by applying the gradient method to problem~\eqref{eq:consensus_optimization_alternative}, each agent $i$, at iteration $\iter$, maintains an estimate $\stateit \in \R^{d}$ about the $i$-th block of a solution to~\eqref{eq:consensus_optimization_alternative} and explicitly updates it according to
\begin{align}
	\stateitp = \stateit - \step \bigg(\const\bigg(\stateit - \dfrac{1}{N}\sum_{j=1}^N \statejt\bigg) + \sum_{j=1}^N \nabla f_j\bigg(\dfrac{1}{N}\sum_{k=1}^N \statekt\bigg)\!\bigg)\!,
	\label{eq:parallel_for_consensus}
\end{align}
where $\step > 0$ is the step size.
Method~\eqref{eq:parallel_for_consensus} can be cast in the form of the meta-algorithm~\eqref{eq:local_centralized_method} by setting
\begin{align*}
	\outi(\statei) &:= \statei
	\\
	\agg(\state) &:= 
	\begin{bmatrix} 
		\agg_1(\state) 
		\\
		\agg_2(\state)
	\end{bmatrix} 
	:= \begin{bmatrix}
		\tfrac{1}{N}\sum_{i=1}^N \statei
		\\[.6em]
		\sum_{i=1}^N \nabla f_i(\tfrac{1}{N}\sum_{j=1}^N \statej)
	\end{bmatrix}
	\\
	\ali(\statei,\agg(\state)) &:= \statei - \step\left(\const(\statei - \agg_1(\state)) + \agg_2(\state)\right)
	.%
\end{align*} 
Given $\statet \!:=\! \col{\statet_1,\dots,\statet_N} \!\in\! \R^{Nd}$, we note that the stacking dynamics arising from the local one~\eqref{eq:parallel_for_consensus} compactly reads as 
\begin{align}
	\statetp = \statet - \step\nabla \tfc(\statet).\label{eq:aggregate_dynamics_for_consensus_optimization}
\end{align}
Further, by assuming that $\sum_{i=1}^N f_i(x)$ is radially unbounded and continuously differentiable with Lipschitz continuous gradients, these properties apply to $\tfc(x)$ too by construction (cf.~\eqref{eq:tilde_f}).
Hence, by using the Descent Lemma~\cite[Prop.~6.1.2]{bertsekas2015convex}, algorithm~\eqref{eq:aggregate_dynamics_for_consensus_optimization} satisfies Assumptions~\ref{ass:centralized_equilibria} and~\ref{ass:centralized_convergence} with sufficiently small $\step$ and by setting
\begin{align*}
	\cSc := \R^{N\n}, 
	\hspace{.2cm}
	\Vc(\state) := \tfc(\state).
\end{align*}
In the case of strong convexity of $\sum_{i=1}^N f_i(x)$, given its unique minimizer $\xstar$, it also satisfies the additional Assumption~\ref{ass:additional_assumption} with 
\begin{align*}
	\Vc(\state) := \tfc(\state) - \tfc(\1\xstar).
\end{align*}
\subsubsection*{Constraint-Coupled Optimization} we now focus on problems in the form~\eqref{eq:constr_coupled}. %
Let $A := \begin{bmatrix}A_1&\hdots&A_N\end{bmatrix}$, $b:=\sum_{i=1}^N b_i$, and $H_\rho: \R \times \R \to \R$ be defined as
\begin{align*}
	&H_\rho(v,\lambda) 
	%\notag\\
	%		&
	:= 
		\begin{cases}
			\hspace{-.01cm} v \lambda + \dfrac{\rho}{2}v^2 \hspace{.23cm} \text{if }  \rho v + \lambda \ge 0
			\\
			\hspace{-.01cm}-\dfrac{1}{2\rho} \lambda^2 \hspace{.59cm} \text{if } \rho v + \lambda < 0,
		\end{cases}
\end{align*}
where $\rho > 0$.
Then, let $\cHp: \R^{\m} \times \R^\m \to \R$ be defined as 
\begin{align}
	\cHp(v,\lambda) := \sum_{\ell=1}^\m H_\rho\left([v]_\ell,[\lambda]_\ell\right).\label{eq:cH}
\end{align}
By slightly adapting the approach in~\cite{qu2018exponential}, we use $\cHp$ to introduce an Augmented Lagrangian function $\cL: \R^{\n} \times \R^{Nm} \to \R$ associated to problem~\eqref{eq:constr_coupled} and defined as 
\begin{align*}
	\cL(x,\lambda) &:= \sum_{i=1}^N f_i(x_i) + \cHp\bigg(Ax-b,\dfrac{1}{N}\sum_{i=1}^N \lambda_i\bigg)
	-\frac{\const}{2} \lambda\T \big(I - \tfrac{1}{N}\1\1\T\big)\lambda,%
\end{align*}
in which $\lambda := \col{\lambda_1,\dots,\lambda_N} \in \R^{Nm}$ with $\lambda_i \in \R^{m}$ for all $i \in \set$ and $\const > 0$ can be arbitrarily chosen.
Finding saddle points of $\cL$ allows for solving problem~\eqref{eq:constr_coupled} with linear rate~\cite{qu2018exponential}.
Hence, for all $i \in \set$, agent $i$ maintains a local solution estimate $\xit \in \R^{\n_i}$ and a local multiplier $\lit \in \R^{\m}$ and updates them with a descent and an ascent step based on $\cL$, namely
\begin{subequations}\label{eq:parallel_for_cc}
	\begin{align}
		\xitp &= \xit -\step\nabla f_i(\xit) 
		-\step A_i\T\nabla_1 \cHp\bigg(A\xt - b, \dfrac{1}{N}\sum_{j=1}^N \ljt\bigg)
		\\
		\litp &= \lit + \step\const\bigg(\dfrac{1}{N}\sum_{j=1}^N \ljt -\lit\bigg)
		+ \step\tfrac{1}{N}\nabla_2 \cHp\bigg(A\xt -b, \dfrac{1}{N}\sum_{j=1}^N \ljt\bigg),
	\end{align}
\end{subequations}
where $\step > 0$ is the step size.
We cast method~\eqref{eq:parallel_for_cc} in the meta-algorithm form~\eqref{eq:local_centralized_method} by setting%
\begin{subequations}\label{eq:g_cc}
	\begin{align}
		&\state_i := \begin{bmatrix}x_i\\ \lambda_i\end{bmatrix}, \quad \outi(\statei) := \begin{bmatrix}I_{\n_i}& 0_\m\end{bmatrix}\statei
		\\
		& \agg(\state) := \begin{bmatrix} 
			\agg_1(\state) 
			\\
			\agg_2(\state) 
		\end{bmatrix} := 
		\begin{bmatrix}
			\sum_{i=1}^N (A_ix_i - b_i)
			\\[.5em]
			\tfrac{1}{N}\sum_{i=1}^N \lambda_i
		\end{bmatrix}
		\\
		&\ali(\statei,\agg(\state))  :=  \state_i + \step\begin{bmatrix}
		- \nabla f_i(x_i) - A_i\T\nabla_{1}\cHp(\agg_1(\state),\agg_2(\state))
		\\
		\const(\agg_2(\state) - \lambda_i) + \tfrac{1}{N}\nabla_2 \cHp(\agg_1(\state),\agg_2(\state))\end{bmatrix},
	\end{align}
\end{subequations}
for all $i \in \set$.
The local update~\eqref{eq:parallel_for_cc} allows for a parallel implementation of the augmented primal-dual method proposed in~\cite{qu2018exponential} to deal with generic optimization problems with linear inequality constraints.
Based on this observation, by assuming that $A$ is full-row rank and $\sum_{i=1}^N \f_i(x_i)$ is strongly convex with Lipschitz continuous gradients, for sufficiently small $\step$, the stacking dynamics arising from~\eqref{eq:parallel_for_cc} satisfies Assumptions~\ref{ass:centralized_equilibria},~\ref{ass:centralized_convergence}, and~\ref{ass:additional_assumption} by adapting~\cite[Lemma~5]{qu2018exponential} and setting
\begin{align*}
	\Vc(\state) &:= 
	\begin{bmatrix}
	   \x - \xstar
	   \\
	   \lambda -  \1\lstar
   \end{bmatrix}\T \! \begin{bmatrix}
	\kappa I& A\T\1\T 
	   \\
	   \1A& \kappa I
   \end{bmatrix}\begin{bmatrix}
	   \x - \xstar
	   \\
	   \lambda - \1\lstar
   \end{bmatrix}
    \\
\cSc &:= \R^{\n} \times \R^{N\m}_+,
\end{align*}
where $\xstar \in \R^{\n}$ and $\lstar \in \R^m$ are the unique solution to problem~\eqref{eq:constr_coupled} and the corresponding multiplier, while $\kappa > 0$ needs to be sufficiently large~\cite[Th.~2]{qu2018exponential}.
In the case of nonlinear constraints, we refer to~\cite{tang2020semi} to get semi-global exponential stability properties for the centralized algorithm~\eqref{eq:parallel_for_cc}.

\subsubsection*{Aggregative Optimization} we now report a \centralized/ method tailored for problems in the form~\eqref{eq:aggregative_optimization_problem}.
Given a step size $\step > 0$, agent $i$, at iteration $\iter$, updates an estimate $\stateit \in \R^{\n_i}$ about the $i$-th block of a problem solution via
\begin{align}
	\stateitp &= \stateit- \step \nabla_1\f_i(\stateit,\sigma(\statet))
	-\step \dfrac{\nabla\phii(\stateit)}{N}
	\sum_{j=1}^N\nabla_2\f_j(\statejt,\sigma(\statet)).
	\label{eq:parallel_for_agg}
\end{align}
We cast method~\eqref{eq:parallel_for_agg} in the meta-algorithm form~\eqref{eq:local_centralized_method} by setting%
\begin{align*}
	\outi(\statei) &:= \statei
	\\
	\agg(\state) &:= \begin{bmatrix} 
		\agg_1(\state)
		\\
		\agg_2(\state)
	\end{bmatrix}
	 :=  \begin{bmatrix}
		\tfrac{1}{N}  \sum_{i=1}^N  \phii(\statei)
		\\
		\tfrac{1}{N}  \sum_{i=1}^N  \nabla_2 \f_i(\statei,\tfrac{1}{N}  \sum_{j=1}^N  \phij(\statej))
	\end{bmatrix}
	\\
	\ali(\statei,\agg(\state)) &:= \statei - \step \left( \nabla_1 \f_i(\statei,\agg_1(\state)) + \nabla\phi(\state)\agg_2(\state)\right),
\end{align*}
for all $i \in \set$.
Let $\fs: \R^{\n} \to \R$ be defined as $\fs(x) := \sum_{i=1}^N \f_i(x_i,\sigma(x))$.
Then, by computing the gradient of $\fs$, we note that the stacking update of~\eqref{eq:parallel_for_agg} reads as 
\begin{align}
	\statetp = \statet - \step\nabla \fs(\statet).\label{eq:aggregate_for_aggregative}
\end{align}
Namely, system~\eqref{eq:aggregate_for_aggregative} corresponds to a parallel implementation of the gradient descent method applied to problem~\eqref{eq:aggregative_optimization_problem}. 
Thus, if $\fs$ is radially unbounded with Lipschitz continuous gradients, for sufficiently small $\step$, system~\eqref{eq:aggregate_for_aggregative} satisfies Assumptions~\ref{ass:centralized_equilibria} and~\ref{ass:centralized_convergence} via Descent Lemma~\cite[Prop.~6.1.2]{bertsekas2015convex} and setting 
\begin{align*}
	\cSc := \R^{\n}, \quad \Vc(\state) := \fs(\state).
\end{align*}
In case of strong convexity of $\fs$, given its unique minimizer $\xstar$, Assumption~\ref{ass:additional_assumption} is satisfied too with $\Vc(\state) := \fs(\state) - \fs(\xstar)$.

\subsubsection*{Aggregative Games} we now provide an equilibrium-seeking scheme for games in the form~\eqref{eq:games}.
Inspired by~\cite{qu2018exponential}, the local constrained problem of agent $i$ can be treated by resorting to the Augmented Lagrangian $\cL_i: \R^{\n} \times \R^{Nm} \to \R$ defined as 
\begin{align*}
	\cL_i(x,\lambda) &:= J_i(x_i,\sigma(x)) + \cHp\bigg(Ax-b,\dfrac{1}{N}\sum_{i=1}^N \lambda_i\bigg)
	-\frac{\const}{2} \lambda\T \big(I - \tfrac{1}{N}\1\1\T\big)\lambda,%
\end{align*}
where $\cHp$ is defined in~\eqref{eq:cH} and $\const > 0$ can be arbitrarily chosen.
As already mentioned above, differently from standard Lagrangian functions, this choice allows for devising algorithms with linear rate~\cite{qu2018exponential}.
In particular, at each $\iter \in \N$, agent $i$ maintains an estimate $\xit \in \R^{n_i}$ about the $i$-th block of a v-GNE of~\eqref{eq:games}, a multiplier $\lit \in \R^{\m}$, and updates them via%
\begin{subequations}\label{eq:parallel_for_games}
	\begin{align}
			\xitp &= \xit - \step \Gi(\xit,\sigma(\xt)) 
			- \step A_i\T\nabla_1 \cHp\bigg(A\xt -b,\dfrac{1}{N}\sum_{j=1}^N \ljt\bigg)
			\\
			\litp &=  \lit + \step \const\bigg(\dfrac{1}{N}\sum_{j=1}^N \ljt - \lit\bigg)  
			+ \step \dfrac{1}{N}\nabla_{2} \cHp\bigg(A\xt -b,\dfrac{1}{N}\sum_{j=1}^N \ljt\bigg),
	\end{align}
\end{subequations}
where $\step > 0$ is the step size, while, for all $i \in \set$, $\Gi: \R^{\n_i} \times \R^{d} \to \R^{\n_i}$ reads as
\begin{align*}
	\Gi(x_i,z_i) &:= 
		\nabla_1 \J_i(x_i,z_i) + \dfrac{\nabla \phi_i(x_i)}{N}\nabla_2 \J_i(x_i,z_i).
\end{align*}
We cast method~\eqref{eq:parallel_for_games} in the meta-algorithm form~\eqref{eq:local_centralized_method} by setting%
\begin{align*}
	\statei &
	:= \col{\x_i,\lambda_i}, 
	\quad 
	\outi(\statei) 
	:= 
	\begin{bmatrix}
		I_{\n_i}& 0_\m
	\end{bmatrix}\statei
	\\
	\agg(\state)
	 &:= 
	\begin{bmatrix} 
		\agg_1(\state) 
		\\
		\agg_2(\state) 
		\\
		\agg_3(\state)
	\end{bmatrix} 
	:= 
	\begin{bmatrix}
		\tfrac{1}{N}\sum_{i=1}^N \phii(x_i) 
		\\[.2em]
		\tfrac{1}{N}\sum_{i=1}^N \lambda_i
		\\[.2em]
		\sum_{i=1}^N (A_i x_i - b_i)
	\end{bmatrix}
	\\
	\ali(\statei,\agg(\state)) &:= \statei + \step
	\begin{bmatrix}
		- \Gi(x_i,\agg_1(\state)) - A_i\T \nabla_1\cHp(\agg_3(\state),\agg_2(\state))
		\\
		\const(\agg_2(\state) - \lambda_i)  +  \tfrac{1}{N}\nabla_{2} \cHp(\agg_3(\state),\agg_2(\state))
	\end{bmatrix}
	.
\end{align*}
By adapting~\cite[Lemma~5]{qu2018exponential}, it can be shown that, for sufficiently small $\step$, dynamics~\eqref{eq:parallel_for_games} matches Assumptions~\ref{ass:centralized_equilibria},~\ref{ass:centralized_convergence}, and~\ref{ass:additional_assumption} by assuming that $\col{\Gg_1(x_1,\sigma(x)),\dots,\Gg_N(x_N,\sigma(x))}$ is strongly monotone, the gradients of $J_i$ and $\phi_i$ are Lipschitz continuous for all $i \in \set$, $A$ is full-row rank, and setting  
\begin{align*}
	\Vc(\state) &:= 
	\begin{bmatrix}
		\x - \xstar
		\\
		\lambda -  \1\lstar
	\end{bmatrix}\T \! \begin{bmatrix}
	 \kappa I& A\T\1\T 
		\\
		\1A& \kappa I
	\end{bmatrix}\begin{bmatrix}
		\x - \xstar
		\\
		\lambda - \1\lstar
	\end{bmatrix}
	\\
	\cSc &:= \R^{\n} \times \R^{N\m}_+,
\end{align*}
where $\xstar \!\in\! \R^{\n}$ and $\lstar \!\in\! \R^{m}$ are the unique v-GNE and multiplier of~\eqref{eq:games}, and $\kappa \! > \! 0$ must be large enough~\cite[Th.~2]{qu2018exponential}.

\subsection{Dynamic Average Consensus}
\label{sec:dynamic_average_consensus}

For the \centralized/ algorithms~\eqref{eq:parallel_for_cc} and~\eqref{eq:parallel_for_games},
it holds 
\begin{align}
	\agg(\state) = \dfrac{1}{N}\sum_{i=1}^N \laggi(\statei),\label{eq:dynamic_averaging_consensus}
\end{align}
where each $\laggi: \R^{\nstatei} \!\to\! \R^{\na}$ gives agent $i$ contribution.
Namely, $\agg(\state)$ is the solution to dynamic average consensus problems (see the tutorial~\cite{kia2019tutorial}).
In the \centralized/ schemes~\eqref{eq:parallel_for_consensus} and~\eqref{eq:parallel_for_agg}, $\agg(\state)$ is the solution to dynamic average consensus problems with composite functions, namely
\begin{align}\label{eq:nested_agg}
	\hspace{-.1cm}\agg(\state) \! :=\!\! \begin{bmatrix}
		\aggI(\state)
		\\
		\aggE(\state,\1\aggI(\state))
	\end{bmatrix} \!\! :=\! \frac{1}{N}\begin{bmatrix}
		\sum_{i=1}^N \laggIi(\state_i) 
		\\
		\sum_{i=1}^N \laggEi(\state_i,\aggI(\state)) 
	\end{bmatrix}\!\!.\!\!
\end{align}
In detail, we split $\agg$ into $\aggI: \R^{\nstate} \to \R^{\naI}$ and $\aggE: \R^{\nstate} \times \R^{N\naI} \to \R^{\naE}$ which, in turn, depend on the local contributions encoded in each $\laggIi: \R^{\nstatei} \to \R^{\naI}$ and $\laggEi: \R^{\nstatei} \times \R^{\naI} \to \R^{\naE}$. 
This ``composite'' structure makes a distinction between the ``internal'' aggregation function $\aggI$ (e.g., the mean of $\state_1$, $\dots, \state_N$) and the ``external'' one $\aggE$
(e.g., a global gradient computed at the mean $\aggI$).
In Appendix~\ref{sec:nested}, we formally show (cf. Proposition~\ref{prop:nested_trackers}) that a cascade structure of suitable dynamic average consensus schemes constitutes a \cnsor dynamics satisfying Assumptions~\ref{ass:orthogonality},~\ref{ass:equilibria_consensus}, and~\ref{ass:tracking} in the case of composite functions as in~\eqref{eq:nested_agg}.
For this reason, we now present a set of dynamic average consensus schemes satisfying Assumptions~\ref{ass:orthogonality},~\ref{ass:equilibria_consensus}, and~\ref{ass:tracking} and, thus, representing possible candidates to constitute the \cnsor part~\eqref{eq:global_update_z} of the overall distributed algorithm.
Nonetheless, we recall that our approach is not limited to dynamic average consensus.
As long as Assumptions~\ref{ass:orthogonality},~\ref{ass:equilibria_consensus}, and~\ref{ass:tracking} are verified, if required by the definition of the specific $\agg$, our approach may also include different consensus protocols, e.g., the leader-follower schemes widely used in game theory~\cite{gadjov2018passivity}.
The schemes are presented by introducing $u_i \!:=\! \laggi(\statei)$, $u_{\cN_i} \!:=\! \col{u_j}_{j \in \cN_i} \!\in\! \R^{\degi\na}$ for all $i \!\in\! \set$,
$u := \col{u_1,\dots,u_N} \in  \R^{N\na}$, and a weighted adjacency matrix $\cW \!\in\! \R^{N \times N}$ whose entries $w_{ij}$ match the sparsity of $\cG$. 

\subsubsection*{Perturbed Consensus Dynamics} this scheme~\cite{zhu2010discrete,kia2019tutorial} has been extensively used in distributed optimization~\cite{nedic2017achieving,daneshmand2020second,li2021distributed,li2021distributedOnline,carnevale2022distributed} and games~\cite{belgioioso2022distributed,carnevale2022tracking}.
In the so-called \emph{causal} form (see~\cite{bin2019system,carnevale2022nonconvex}), agent $i$ sets the approximation function as $\taggi(u_i,\zit) = u_i + \zit$ with $\zit \in \R^{\na}$ updated according to
\begin{align}
	\label{eq:perturbed_consensus_dynamics}
	\zitp &= \sum_{j \in \cN_i} w_{ij}\left(\zjt + u_j\right) - u_i.
\end{align}
System~\eqref{eq:perturbed_consensus_dynamics} matches Assumptions~\ref{ass:orthogonality},~\ref{ass:equilibria_consensus}, and~\ref{ass:tracking} when $\cG$ is strongly connected and $\cW$ is doubly-stochastic~\cite[Th.5]{kia2019tutorial}.

\subsubsection*{Proportional Integral Action for Dynamic Average Consensus} the distributed method~\eqref{eq:perturbed_consensus_dynamics} requires a specific initialization into $\left\{\z \in \R^{N\na} \mid \1\T\z = 0\right\}$~\cite{kia2019tutorial}.
To overcome this limitation, one can discretize the continuous-time scheme in~\cite{freeman2006stability} obtaining
\begin{subequations}\label{eq:perturbed_consensus_dynamics_with_integral_action}
	\begin{align}
		\pitp &= (1-\gamma)\pit - k_P\sum_{j \in \cN_i}w_{ij}(\pit - \pjt) 
		+ k_I\sum_{j \in \cN_i}w_{ij}(\qit - \qjt) + \gamma u_i
		\\
		\qitp &= \qit + k_I\sum_{j \in \cN_i}w_{ij}(\pit - \pjt),
	\end{align}
\end{subequations}
where $\pit, \qit \in \R^{\na}$ and $\gamma, k_P, k_I > 0$ are tuning parameters.
In this scheme, the approximation function is $\taggi(u_i,\col{\pit,\qit}) = \pit$.
If $\cG$ is strongly connected and $\cW$ is doubly-stochastic, the result~\cite[Th.~5]{freeman2006stability}, adapted to discrete-time, ensures that~\eqref{eq:perturbed_consensus_dynamics_with_integral_action} matches Assumptions~\ref{ass:orthogonality},~\ref{ass:equilibria_consensus}, and~\ref{ass:tracking} with suitable $k_P$, $k_I$, $\gamma$.  %

\subsubsection*{R-ADMM for Dynamic Average Consensus} here, we report the distributed method proposed in~\cite{bastianello2022admm} based on Alternating Dual Multiplier Method (ADMM).
In~\cite{carnevale2023admm,carnevale2023distributed,sebastian2024accelerated}, this method has been combined with a gradient-based policy to address consensus optimization.
In detail, the original dynamic consensus problem is turned into the quadratic program
\begin{align}\label{eq:cns_as_opt}
	\begin{split}
		\min_{
				(\s_1,\dots,\s_N) \in \R^{N\na}
		}
		&\sum_{i=1}^N \dfrac{1}{2}\norm{\s_i - u_i}^2
		\\
		\text{s.t.:} \: \: &\s_i = \s_j,  \: \forall (i,j)\in \cE.
	\end{split}
\end{align}
If graph $\cG$ is connected, the (unique) optimal solution $\s\ud\star \in \R^{N\na}$ to~\eqref{eq:cns_as_opt} reads as $\s\ud\star = \1\tfrac{1}{N}\sum_{j=1}^N u_j$~\cite{bastianello2022admm}. 
From this observation, the dynamic average consensus is achieved by addressing~\eqref{eq:cns_as_opt} via Distributed R-ADMM~\cite{bastianello2020asynchronous}.
Namely, agent $i$ maintains a variable $\zijt \in \R^{\na}$ for each neighbor $j \in \cN_i$ and implements%
	\begin{align*}
		\taggi(u_i,\zit)
		&= \argmin_{\s_i \in \R^\na}
		\bigg\{\tfrac{1}{2}\norm{\s_i - u_i}^2 \! - \! \s_i\T \sum_{j \in \cN_i}\zijt 
		\! + \! \tfrac{\rho \degi}{2}\norm{\s_i}^2\bigg\}
		\\
		\zijtp &= (1 - \beta)\zijt + \beta\left(-\zjit + 2\rho\taggj(u_j,\zjt)\right),
	\end{align*}
with $\rho > 0$ and $\beta \in (0,1)$ and $\zit := \col{\zijt}_{j \in \cN_i}$.
Since the cost $\tfrac{1}{2}\norm{\s_i - u_i}^2$ is quadratic, the above update reduces to
\begin{subequations}\label{eq:ADMM_update}
	\begin{align}
		\taggi(u_i,\zit) &= \dfrac{1}{1+\rho \degi}\bigg(u_i + \sum_{j\in\cN_i} \zijt\bigg)
		\\
		\zijtp &= (1 - \beta)\zijt + \beta\left(-\zjit + 2\rho \taggj(u_j,\zjt)\right).
	\end{align}
\end{subequations}
If graph $\cG$ is undirected and connected, system~\eqref{eq:ADMM_update} matches Assumptions~\ref{ass:orthogonality},~\ref{ass:equilibria_consensus}, and~\ref{ass:tracking} (see~\cite[Prop.~5]{bastianello2022admm} or~\cite[Lemma~IV.3]{carnevale2023admm}).

\section{Novel Distributed Algorithm \\for Constraint-Coupled Optimization}
\label{sec:application}

In this section, we show the potential of our systematic design by synthesizing a novel distributed algorithm for constraint-coupled problems~\eqref{eq:constr_coupled}.
Then, we provide numerical simulations corroborating the theoretical results.
With this example, we effectively demonstrate the potential of the proposed procedure as it immediately offers both the design and analysis of a distributed algorithm able to recover the recent result in~\cite{qu2018exponential} about (centralized) optimization with inequality constraints.

\subsection{Distributed Algorithm Design}
\label{sec:application_design}

As claimed in Section~\ref{sec:federated_algorithm}, the constraint-coupled setup~\eqref{eq:constr_coupled} can be addressed through the \centralized/ method~\eqref{eq:parallel_for_cc}.
The execution of the considered \centralized/ algorithm needs the aggregation function $\agg(\state)$ specified in~\eqref{eq:g_cc}.
Once the \centralized/ algorithm is chosen and its aggregation function has been identified, we follow the systematic design outlined in Section~\ref{sec:systematic_design} to get its distributed counterpart.
Specifically, for each agent $i \in \set$, we introduce the auxiliary variables $\zit \in \R^{2\m}$ and choose the perturbed consensus dynamics in the causal form~\eqref{eq:perturbed_consensus_dynamics} to determine their evolution.
Then, by decomposing each $\zit$ according to $\zit := \col{\wit,\zetait}$ with $\wit,\zetait \in \R^{\m}$, the \centralized/ method~\eqref{eq:parallel_for_cc} gives rise to the distributed one 
\begin{subequations}\label{eq:new_distributed_algorithm}
\begin{align}
	\xitp &= \xit - \pr\step\nabla f_i(\xit) 
	- \pr\step A_i\T\nabla_1\cHp\left(N(A_i\xit - b_i) + \zetait, \lit + \wit\right)
	\\
	\litp &= \lit  +  \pr\step\const\wit  
	+   \pr\step\tfrac{1}{N}\nabla_2 \cHp\left(N(A_i\xit  -  b_i)  +  \zetait, \lit  +  \wit\right)
	\\
	\witp&= \sum_{j \in \cN_i} w_{ij}\left(\wjt + \ljt\right)
	- \lit
	\\
	\zetaitp &=\sum_{j \in \cN_i} w_{ij}\left(\zetajt  +  N(A_j\xjt  -  b_j)\right)
		\! - \!  N(A_i\xit  \! - \!  b_i).\!
\end{align}
\end{subequations}
Next, we ensure that~\eqref{eq:new_distributed_algorithm} solves problem~\eqref{eq:constr_coupled} with a linear rate.
As far as we know, existing distributed schemes (see, e.g.,~\cite{li2022implicit}) achieve this property only in the case of equality constraints and, thus, this is the first distributed scheme with a linear rate in problems with coupling inequality constraints.
\begin{theorem}\label{th:new_algorithm}
	Consider~\eqref{eq:new_distributed_algorithm}.
	Assume that $\sum_{i=1}^N f_i$ is $\mu$-strongly convex, $\nabla f_i$ is $\lipp$-Lipschitz continuous for all $i \in \set$, and $\kappa_1 I_{\m} \leq AA\T \leq \kappa_2 I_{\m}$ for some $\mu, \lipp, \kappa_1, \kappa_2 > 0$.
	Assume that $\cG$ is strongly connected and $\cW$ doubly stochastic.
	Then, there exist $\bar{\pr}, \bar{\step}, a_1, a_2 > 0$ such that, for all $(\x_i\ud0,\lambda_i\ud0$,$\w_i\ud0$,$\z_i\ud0) \in \R^{\n_i} \times \R^{\m} \times \R^{\m} \times \R^\m$ with $\w_i\ud0 = \z_i\ud0 = 0$ for all $i \in \set$, $\pr \in (0,\bar{\pr})$, and $\step \in (0,\bar{\step})$, it holds 
	\begin{align}
		\norm{\xt - \xstar} \leq a_1\exp(-a_2\iter),\label{eq:convergence_new_algorithm}
	\end{align}
	where $\xstar$ is the unique solution to problem~\eqref{eq:constr_coupled}.
\end{theorem}
\begin{proof}
	The proof directly uses Theorem~\ref{th:main}.
	Indeed, under the enforced assumptions, there exists $\bar{\step} > 0$ such that, for all $\step \in (0,\bar{\step})$, the inspiring \centralized/ scheme~\eqref{eq:parallel_for_cc} satisfies Assumptions~\ref{ass:centralized_equilibria},~\ref{ass:centralized_convergence}, and~\ref{ass:additional_assumption} by adapting~\cite[Lemma~5]{qu2018exponential}.
	Further, the perturbed consensus dynamics (i.e.,~\eqref{eq:perturbed_consensus_dynamics}) of $\col{\wit,\zetait}$ matches Assumptions~\ref{ass:orthogonality},~\ref{ass:equilibria_consensus}, and~\ref{ass:tracking}~\cite[Th.~5]{kia2019tutorial}.
	Thus, the requirements of Theorem~\ref{th:main} are met and~\eqref{eq:convergence_new_algorithm} follows by~\eqref{eq:additional_result}.
\end{proof}

\subsection{Numerical Simulations}
\label{sec:application_numerical_simulations}

We consider an instance of problem~\eqref{eq:constr_coupled} with $N = 10$, $\m \!=\! 2$, $\n_i \!=\! 2$ for all $i \!\in\! \set$, and quadratic costs $f_i: \R^2 \!\to\! \R$ defined as
\begin{align*}
	f_i(x_i) = \dfrac{1}{2}x_i\T Q_ix_i + r_i\T x_i,
\end{align*}
where $Q_i = Q_i\T \in \R^{2 \times 2}$ and $r_i \in \R^{2}$.
For all $i \in \set$, we generate random matrices $Q_i$ with eigenvalues in the set $(0,1)$. 
We randomly select the components of $r_i$ according to a Gaussian distribution with zero mean and variance $1$ and the components of $b_i$ extracting them from the interval $(0,1)$ with a uniform probability distribution.
Besides, we ensure that the matrices $A_i$ guarantee the full-row rankness of $A = \begin{bmatrix}
	A_1& \hdots& A_N
\end{bmatrix}$.
As for the inter-agent communication, we randomly generate an \er/ graph with connectivity parameter $0.3$.
As for the algorithm parameters, we empirically tune them as $\pr = 0.1$, $\step = 0.1$, $\rho = 0.9$, and $\const = 1$.
Fig.~\ref{fig:error} shows the evolution of both $\norm{\statet - \statestar}$ and $\norm{\Tp(\zt) - \pzeq(\statet)}$, where $\statestar \in \R^{\n + N\m}$ contains both the unique problem solution $\xstar$ and the associated multiplier $\lstar$%
, the map $\Tp: \R^{2N\m} \to \R^{2(N-1)\m}$ (see Assumption~\ref{ass:orthogonality}) reads as $\Tp(\z) = T_\perp\z$, where $T_\perp \in \R^{2(N-1)\m \times 2N\m}$ is a matrix whose columns span the subspace orthogonal to the span of $\1_{N,2\m}/\sqrt{N}$, while $\pzeq(\state):=-T_\perp\col{\lambda_1,N(A_1\x_1-b_1),\dots,\lambda_N,N(A_N\x_N-b_N)} \in \R^{2(N-1)\m}$.
For further details about these elements, we point to the reference~\cite{kia2019tutorial}.
As predicted by Theorem~\ref{th:new_algorithm}, Fig.~\ref{fig:error} confirms that~\eqref{eq:new_distributed_algorithm} exhibits exponential convergence.
Fig.~\ref{fig:error} also graphically highlights the different convergence rates of the slow and fast parts of~\eqref{eq:new_distributed_algorithm}.
\begin{figure}[htpb]
	\centering
	\includegraphics[scale=1]{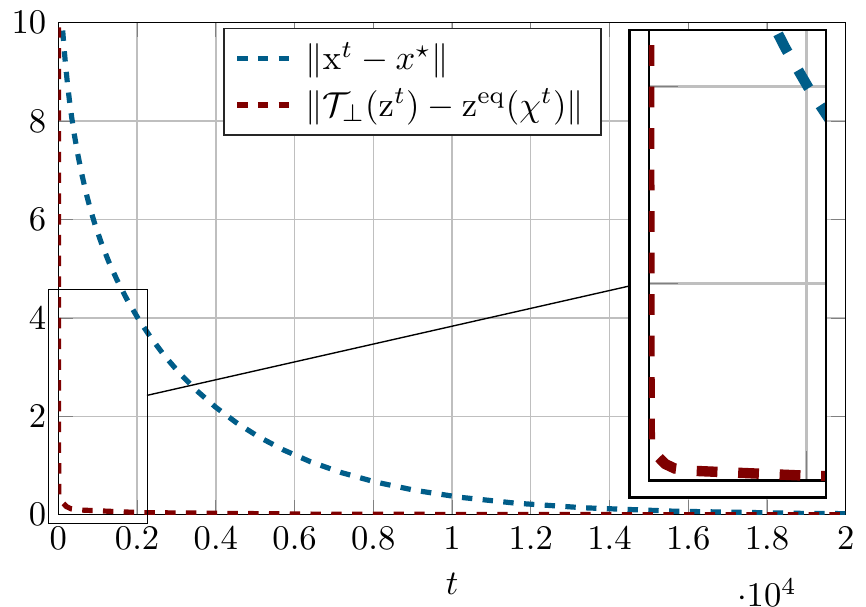}
	\caption{Algorithm~\eqref{eq:new_distributed_algorithm}: evolution over iterations $\iter$ of the quantities $\norm{\statet - \statestar}$ and $\norm{\Tp(\zt) - \pzeq(\statet)}$.}
	\label{fig:error}
\end{figure}

We now give insights on the role of the small parameter $\pr$.
We consider the same setting of the previous simulation and run the centralized method~\eqref{eq:parallel_for_cc} with $\step = 0.1$, $\rho = 0.9$, and $\const = 1$ as well as its distributed counterpart~\eqref{eq:new_distributed_algorithm} with $\step \!=\! 0.1$, $\rho \!=\! 0.9$, $\const = 1$, and different values of $\pr$.
Fig.~\ref{fig:comparison} shows the evolution of the optimality error $\norm{\xt - \xstar}$ in the described cases.
This simulation confirms the validity of our singular-perturbations-based interpretation, as it shows that the distributed scheme~\eqref{eq:new_distributed_algorithm} regains effectiveness only with sufficiently small $\pr$.
We note that by running~\eqref{eq:new_distributed_algorithm} with $\pr =1$ and, thus, with the same step size $\step$ making the centralized scheme~\eqref{eq:parallel_for_cc} effective, we obtain an unstable behavior.
Thus, this test suggests that the necessity to reduce the speed (i.e., to include a sufficiently small $\pr$) is a fundamental limitation of the feedback interconnection architecture required to mimic a centralized scheme in a distributed setting.
\begin{figure}[htpb]
	\centering
	\includegraphics[scale=1]{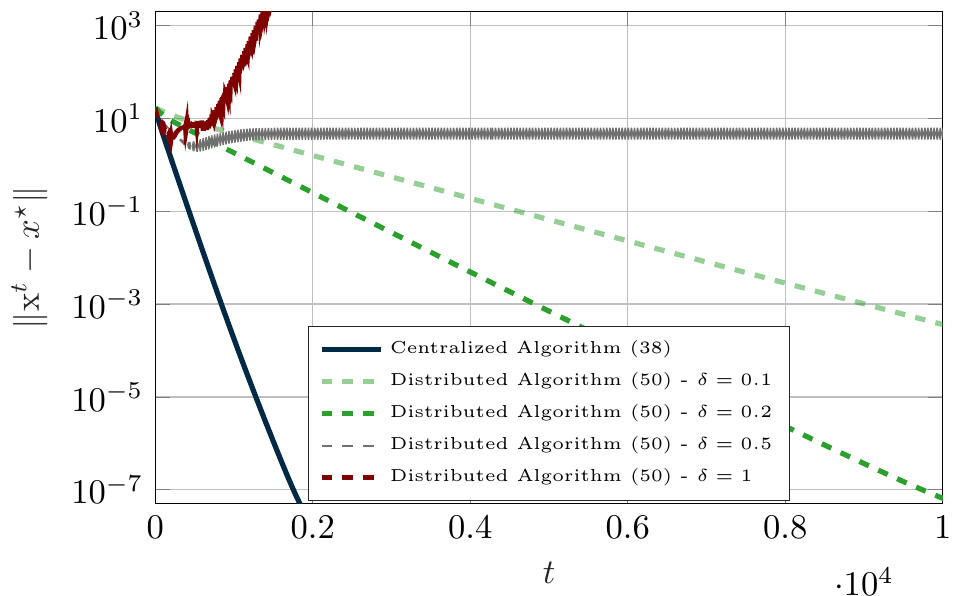}
	\caption{Comparison between the centralized method~\eqref{eq:parallel_for_cc} and its distributed counterpart with different values of $\pr$.}
	\label{fig:comparison}
\end{figure}

\section{Conclusions}
\label{sec:conclusions}

We developed a novel systematic approach for designing and analyzing distributed algorithms for optimization and games over networks.
The proposed approach starts from a \centralized/ optimization-oriented method, identifies an unavailable aggregation function $\agg(\state)$ needed for its execution, and selects a \cnsor dynamics to locally reconstruct it in a distributed way.
By considering the optimization- and \cnsor parts as building blocks, our method prescribes how to interlace them and establishes a set of conditions regarding the two schemes separately considered.
We have shown that these conditions allow for claiming the stability of the obtained distributed algorithm by leveraging on singular perturbations arguments.
Finally, following the proposed approach, we developed a novel distributed scheme for constraint-coupled problems and proved its linear convergence properties.

\appendix

\setcounter{section}{0}
\renewcommand{\thesection}{\Alph{subsection}}

\subsection{A LaSalle's Invariance Principle for SP Systems}
\label{sec:SP}

Here, we provide a result for a special class of SP system in which the slow subsystem has LaSalle convergence properties.
It slightly extends~\cite[Th.~1]{carnevale2022nonconvex} by allowing assumptions to hold in generic sets $\cSc$ and $\cSze$ instead in the entire $\R^{n}$ and $\R^{m}$.
\begin{theorem}\label{th:generic}
	Let $\statet \in \R^{n}$, $\zetat \in \R^m$ and consider the system
	\begin{subequations}\label{eq:interconnected_system}
		\begin{align}
			\statetp &= \statet + \pr\slow(\statet,\zetat)\label{eq:slow_system}
			\\
			\zetatp &= \fast(\zetat,\statet),\label{eq:fast_system}
		\end{align}
	\end{subequations}	
	with $\slow: \R^n \times \R^m \to \R^n$, $\fast: \R^m \times \R^n \to \R^m$, and $\pr > 0$.
	Assume that $\slow$ and $\fast$ are Lipschitz continuous with parameters $\lipp_\slow, \lipp_\fast > 0$, respectively.
	Assume there exist forward-invariant sets $\cSc \subseteq \R^n$ and $\cSze \subseteq \R^{m}$ for~\eqref{eq:slow_system} and~\eqref{eq:fast_system}, namely 
	\begin{subequations}\label{eq:invariance}
		\begin{align}
			\state + \pr \slow(\state,\zeta) \in \cSc \quad \text{for all } \state\in \cSc, \zeta \in \cSze
			\\
			\fast(\zeta,\state) \in \cSze \quad \text{for all } \state \in \cSc, \zeta \in \cSze.
		\end{align}
	\end{subequations} 
	Further, assume that there exists $\zetaeq: \R^n \to \cSze$ such that $\fast(\zetaeq(\state),\state) = \zetaeq(\state)$ for all $\state \in \cSc$ and that $\zetaeq$ is Lipschitz continuous with parameter $\lippeq> 0$.
	Let
	\begin{align}\label{eq:reduced_system_generic}
		\statetp = \statet + \pr \slow(\statet,\zetaeq(\statet))
	\end{align}
	be the reduced system and, given an arbitrary $\state \in \R^{n}$, let
	\begin{align}\label{eq:boundary_layer_system_generic}
		\tzetatp = \fast(\tzetat + \zetaeq(\state),\state) - \zetaeq(\state)
	\end{align}
	be the boundary layer system with $\tzetat \in \R^m$. 
	Assume there exist $\Vzeta: \R^{m} \to \R$ and $b_1, b_2, b_3, b_4 > 0$ such that%
	\begin{subequations}\label{eq:Vz}
		\begin{align}
			&b_1\|\tzeta\|^2 \leq \Vzeta(\tzeta) \leq b_2\|\tzeta\|^2\label{eq:Vz_quadratic_bound}
			\\
			&\Vzeta(\fast(\tzeta \! + \! \zetaeq(\state),\state) \! - \! \zetaeq(\state)) \! - \! \Vzeta(\tzeta) \leq - b_3\|\tzeta\|^2 \label{eq:Vz_minus_theorem}
			\\
			&|\Vzeta(\tzeta) - \Vzeta(\tzeta^\prime)| \leq b_4\|\tzeta - \tzeta^\prime\|(\|\tzeta\| + \|\tzeta^\prime\|),\label{eq:Vz_bound_theorem}
		\end{align}
	\end{subequations}
	for all $\state \in \cSc$, $(\tzeta + \zetaeq(\state)) \in \cSze$, and $(\tzeta^\prime + \zetaeq(\state)) \in \cSze$.
	Further, assume there exist $\bar{\pr}_1, c_1, \cf, \cff, \cfff \! > \! 0$, a radially unbounded and continuously differentiable function $\Vc:$ $\R^n \!\to\! \R$, and a function $\dd:\R^{n} \!\to\! \R_+$ such that, for all $\pr \!\in\! (0,\bar{\pr}_1)$, it holds%
	\begin{subequations}\label{eq:Vc_theorem}
		\begin{align}
			\Vc(\state +\pr \slow(\state,\zetaeq(\state))) \! - \! \Vc(\state) & \leq -  \pr c_1\dd(\state)^2\!\!\!\label{eq:Vc_minus}
			\\
			\norm{\nabla \Vc(\state)} &\leq \cf\dd(\state)
			\label{eq:nabla_vc_d}
			\\
			\norm{\nabla \Vc(\state) - \nabla\Vc(\state^\prime)} &\leq \cff\norm{\state - \state^\prime}
			\label{eq:lipschitz_vc_app}
			\\
			\norm{\slow(\state,\zetaeq(\state))} &\leq \cfff\dd(\state),\label{eq:slow_dd}
		\end{align}
	\end{subequations}
	for all $\state, \state^\prime \in \cSc$. %
	Then, there exist $\bar{\pr} \in (0,\bar{\pr}_1)$, $\psi > 0$ such that, for any $\pr \in (0,\bar{\pr})$, the trajectories of~\eqref{eq:interconnected_system} satisfy
	\begin{align}
		&V(\statetp,\zetatp - \zetaeq(\statetp)) 
		- V(\statet,\zetat - \zetaeq(\statet)) 
		\leq 
		-\psi\left(\dd(\statet)^2 + \norm{\zetat - \zetaeq(\statet)}^2\right),
		\label{eq:result_generic}%
	\end{align}
	for all $(\state\ud0,\zeta\ud0) \!\in\! \cSc \!\times\!\cSze$, and $\iter \!\in\! \N$, where $V\!\!: \R^{n} \!\times\! \R^{m} \!\to\! \R$ is
	\begin{align} 
		V(\state,\tzeta) := \Vz(\tzeta) + \Vc(\state).\label{eq:V}
	\end{align}
\end{theorem}
\begin{proof}
	We note that $(\state\ud0,\zeta\ud0) \in \cSc \times \cSze$ and, thus, in light of~\eqref{eq:invariance}, it holds $(\statet,\zetat) \in \cSc \times \cSze$ for all $\iter \in \N$.
	Now, we define
	$\tzetat := \zetat - \zetaeq(\statet)$,
	and rewrite~\eqref{eq:interconnected_system} as 
	\begin{subequations}\label{eq:interconnected_system_error}
		\begin{align}
			\statetp &= \statet + \pr\slow(\statet,\tzetat+\zetaeq(\statet))\label{eq:slow_system_error}
			\\
			\tzetatp &= \fast(\tzetat + \zetaeq(\statet),\statet) - \zetaeq(\statetp).\label{eq:fast_system_error}
		\end{align}
	\end{subequations}	
	Pick the function $\Vc$ satisfying~\eqref{eq:Vc_theorem}. 
	Thus, by evaluating $\Delta \Vc(\statet) := \Vc(\statetp) - \Vc(\statet)$ along~\eqref{eq:slow_system_error}, we get 
	\begin{align}
		\Delta \Vc(\statet) 
		&=
		\Vc(\statet+\pr\slow(\statet,\tzetat+\zetaeq(\statet))) - \Vc(\statet)
		\notag\\
		&\stackrel{(a)}{=}
		\Vc(\statet+\pr\slow(\statet,\zetaeq(\statet))) - \Vc(\statet) 
		\notag\\
		&
		\hspace{.4cm} 
		+  \Vc(\statet  +  \pr\slow(\statet,\tzetat  +  \zetaeq(\statet))) 
		 -  \Vc(\statet  +  \pr\slow(\statet,\zetaeq(\statet)))
		\notag\\
		&\stackrel{(b)}{\leq}
		-\pr c_1\dd(\statet)^2
		+  \Vc(\statet  +  \pr\slow(\statet,\tzetat  +  \zetaeq(\statet))) 
		 -  \Vc(\statet  +   \pr\slow(\statet,\zetaeq(\statet)))\label{eq:Delta\Vc_int}
		,
	\end{align}
	where in $(a)$ we add and subtract $\Vc(\statet + \pr\slow(\statet,\zetaeq(\statet)))$ and in $(b)$ we use~\eqref{eq:Vc_minus} to bound $\Vc(\statet+\pr\slow(\statet,\zetaeq(\statet))) - \Vc(\statet)$.
	In light of the Lipschitz continuity of $\Vc$ (cf.~\eqref{eq:lipschitz_vc_app}), Descent Lemma~\cite[Prop.~6.1.2]{bertsekas2015convex} and the bound~\eqref{eq:nabla_vc_d} ensure that
	\begin{align}
		\Vc(\state_1+\state_2) - \Vc(\state_1+\state_3)
		&
		 \leq  \cf\dd(\state_1)\norm{\state_2  - \state_3} 
		+\tfrac{\cff}{2} \big(\norm{\state_2}^2  +  \norm{\state_3}^2\big),
		\label{eq:Vc_bound}
	\end{align}
	for all $\state_1, \state_2, \state_3 \! \in \! \cSc$, which allows us to bound~\eqref{eq:Delta\Vc_int} as 
	\begin{align}
		\Delta \Vc(\statet)
		&\leq 
		-\pr c_1\dd(\statet)^2 
		+  \pr \cf \dd(\statet) \|\slow(\statet,\tzetat +  \zetaeq(\statet)) -  \slow(\statet,\zetaeq(\statet))\|
		\notag\\
		&
		\hspace{.4cm}
		+ \pr^2\tfrac{\cff}{2} \left(\|\slow(\statet,\tzetat + \zetaeq(\statet))\|^2   +    \|\slow(\statet,\zetaeq(\statet))\|^2\right).\label{eq:Delta\Vc}
	\end{align}
	Now, we add and subtract $\slow(\statet,\zetaeq(\statet))$ into $\|\slow(\statet,\tzetat + \zetaeq(\statet))\|^2$ and use the bound~\eqref{eq:slow_dd} in~\eqref{eq:Delta\Vc} to get
	\begin{align}
		\Delta \Vc(\statet)
		&\leq
		-\pr (c_1 - \pr \tfrac{\cff\cfff^2}{2}) \dd(\statet)^2
		+ \pr \cf \dd(\statet)\|\slow(\statet,\tzetat  +  \zetaeq(\statet))  -  \slow(\statet,\zetaeq(\statet))\|
		\notag\\
		&\hspace{.4cm}
		+ \pr^2\tfrac{\cff}{2}\|\slow(\statet,\tzetat  +  \zetaeq(\statet))  -  \slow(\statet,\zetaeq(\statet))  +  \slow(\statet,\zetaeq(\statet))\|^2
		\notag\\
		&\stackrel{(a)}{\leq}
		-\pr (c_1 - \pr \cff\cfff^2)\dd(\statet)^2
		+ \pr \cf\lipp_{\slow}\dd(\statet)\|\tzetat\|
		+ \pr^2\cff\lipp_{\slow}^2\|\tzetat\|^2 
		+ \pr^2\cff\cfff\lipp_{\slow}\|\tzetat\|\dd(\statet),\label{eq:Delta\Vc_final}
	\end{align}
	where $(a)$ exploits the Lipschitz continuity of $\slow$, the triangle and Young's inequalities, and~\eqref{eq:slow_dd}.
	Now, by evaluating $\Delta \Vzeta(\tzetat) \!:=\! \Vzeta(\tzetatp) \!-\! \Vzeta(\tzetat)$ along the trajectories~\eqref{eq:fast_system_error}, we get
	\begin{align}
		\Delta \Vzeta(\tzetat) &= \Vzeta(\fast(\tzetat + \zetaeq(\statet),\statet) - \zetaeq(\statetp)) - \Vzeta(\tzetat)
		\notag\\
		&\stackrel{(a)}{=} \Vzeta(\fast(\tzetat+\zetaeq(\statet),\statet)-\zetaeq(\statet)) - \Vzeta(\tzetat) 
		\notag\\
		&\hspace{.4cm}
		+ \Vzeta(\fast(\tzetat + \zetaeq(\statet),\statet) - \zetaeq(\statetp)) 
		-  \Vzeta(\fast(\tzetat+\zetaeq(\statet),\statet)-\zetaeq(\statet))
		\notag\\
		&\stackrel{(b)}{\leq}
		-b_3\|\tzetat\|^2 + \tilde{\Vzeta}(\tzetat,\statet),\label{eq:\Vzeta_before_tilde_\Vzeta}
	\end{align}
	where in $(a)$ we add and subtract the term $\Vzeta(\fast(\tzetat+\zetaeq(\statet),\statet)-\zetaeq(\statet))$ and in $(b)$ we bound $\Vzeta(\fast(\tzetat+\zetaeq(\statet),\statet)-\zetaeq(\statet)) - \Vzeta(\tzetat)$ by using~\eqref{eq:Vz_minus_theorem} and introduce
	\begin{align*}
		\tilde{\Vzeta}(\tzetat,\statet) &:= \Vzeta(\fast(\tzetat+\zetaeq(\statet),\statet)-\zetaeq(\statetp)) 
		- \Vzeta(\fast(\tzetat+\zetaeq(\statet),\statet)-\zetaeq(\statet)).
	\end{align*}
	By using~\eqref{eq:Vz_bound_theorem}, we bound the term $\tilde{\Vzeta}(\tzetat,\statet)$ as
	\begin{align}
		\tilde{\Vzeta}(\tzetat,\statet)
		&\leq  b_4 \norm{\zetaeq(\statetp)  -  \zetaeq(\statet)}\|\fast(\tzetat  +  \zetaeq(\statet),\statet)  - \zetaeq(\statetp)\|
		\notag\\
		&\hspace{.4cm}
		+b_4  \norm{\zetaeq(\statetp)  -  \zetaeq(\statet)}\|\fast(\tzetat  +  \zetaeq(\statet),\statet)  - \zetaeq(\statet)\|
		\notag\\
		&\stackrel{(a)}{\leq}
		b_4 \norm{\zetaeq(\statetp)  -  \zetaeq(\statet)}^2
		+  b_4 2 \norm{\zetaeq(\statetp)  -  \zetaeq(\statet)}\|\fast(\tzetat  \! + \!  \zetaeq(\statet),\statet)  \! - \! \zetaeq(\statet)\|
		\notag\\
		&\stackrel{(b)}{\leq}
		\pr^2b_4\lippeq^2\|\slow(\statet,\tzetat + \zetaeq(\statet))\|^2
		+  \pr b_42\lippeq\lipp_{\fast}\|\slow(\statet,\tzetat + \zetaeq(\statet))\|\|\tzetat\|,\label{eq:tilde\Vzeta}
	\end{align}
	where in $(a)$ we add within the second norm $\pm\zetaeq(\state)$ and use the triangle inequality, while in $(b)$ we use $\fast(\zetaeq(\statet),\statet) = \zetaeq(\statet)$ 
	and exploit the Lipschitz continuity of $\zetaeq$ and $\fast$. 
	Now, we add and subtract $\slow(\statet,\zetaeq(\statet))$ in $\|\slow(\statet,\tzetat + \zetaeq(\statet))\|^2$ and $\|\slow(\statet,\tzetat + \zetaeq(\statet))\|$, then we use the triangle inequality, the Lipschitz property of $\slow$, and~\eqref{eq:slow_dd} to bound~\eqref{eq:tilde\Vzeta} as 
	\begin{align}
		\tilde{\Vzeta}(\tzetat,\statet) 
		&
		\leq\pr^2b_4\lippeq^2\lipp_{\slow}^2\|\tzetat\|^2 +\pr^2\cfff^2b_4\lippeq^2\dd(\statet)^2
		+ \pr^2b_4\cfff2\lippeq^2\lipp_{\slow}\|\tzetat\|\dd(\statet)
		+  \pr b_42\lippeq\lipp_{\fast}\lipp_{\slow}\|\tzetat\|^2
		+  \pr b_4\cfff2\lippeq\lipp_{\fast}\dd(\statet)\|\tzetat\|.
		\label{eq:tilde\Vzeta_final}
	\end{align}
	We then use~\eqref{eq:tilde\Vzeta_final} to further bound~\eqref{eq:\Vzeta_before_tilde_\Vzeta} as 
	\begin{align}
		\Delta \Vzeta(\tzetat) 
		&
		\leq - b_3\|\tzetat\|^2 + \pr^2b_4\lippeq^2\lipp_{\slow}^2\|\tzetat\|^2 
		+\pr^2\cfff^2 b_4\lippeq^2\dd(\statet)^2
		\notag
		\\
		&\hspace{.4cm}
		+ \pr^2b_4\cfff 2\lippeq^2\lipp_{\slow}\|\tzetat\|\dd(\statet)
		+  \pr b_42\lippeq\lipp_{\fast}\lipp_{\slow}\|\tzetat\|^2
		+  \pr b_4\cfff 2\lippeq\lipp_{\fast}\dd(\statet)\|\tzetat\|.\label{eq:Delta\Vzeta_final}
	\end{align}
	Now, let us consider $V$ as defined in~\eqref{eq:V}.
	Thus, by evaluating $\Delta V(\statet,\tzetat) \!:=\! V(\statetp,\tzetatp) \!-\! V(\statet,\tzetat) \!=\! \Delta \Vc(\statet) \!+\! \Delta \Vzeta(\tzetat)$ along the trajectories of~\eqref{eq:interconnected_system_error}, we use~\eqref{eq:Delta\Vc_final} and~\eqref{eq:Delta\Vzeta_final} to write
	\begin{align}
		\Delta V  (\statet,\tzetat)  \leq 
		 -  \begin{bmatrix}
			\dd(\statet)\\
			\|\tzetat\|
		\end{bmatrix}\T
		 \Psi(\pr)
		\begin{bmatrix}
			\dd(\statet)
			\\
			\|\tzetat\|
		\end{bmatrix},\label{eq:H_theorem}
	\end{align}
	for all $\iter\in\N$, where $\Psi(\pr) = \Psi(\pr)\T \in \R^{2\times 2}$ is defined as
	\begin{align*}
		\Psi(\pr) := \begin{bmatrix}
			\pr c_1-\pr^2k_1& -\pr k_2 - \pr^2 k_3
			\\
			-\pr k_2 - \pr^2 k_3& b_3 -\pr k_4-\pr^2 k_5
		\end{bmatrix},
	\end{align*}
	in which the notation has been shortened through the constants
	\begin{align*}
		k_1&:=\cfff^2(b_4\lippeq^2+\cff), 
		\hspace{0.1cm}
		k_2:=\tfrac{\cfff(\cf\lipp_{\slow}+b_42\lipp_{\fast}\lippeq)}{2} %
		\notag\\
		k_3&:=\tfrac{\cff\lipp_{\slow} + b_4\cfff2\lipp_{\slow}\lippeq^2}{2}
		\\
		k_4&:=b_42\lipp_{\slow}\lipp_{\fast}\lippeq, \hspace{0.18cm}
		k_5:=\cff\lipp_{\slow}^2+b_4\lipp_{\slow}^2\lippeq^2.
	\end{align*}
	Being $\Psi(\pr) = \Psi(\pr)\T$, it holds $\Psi(\pr) > 0$ if and only if 
	\begin{align}\label{eq:condition_theorem}
		\begin{cases}
			\pr c_1 > p_1(\pr)
			\\
			\pr c_1b_3 > p_2(\pr),
		\end{cases}
	\end{align}
	where we have introduced the polynomials $p_1(\pr) := \pr^2k_1$ and
		\begin{align*}
			p_2(\pr)\! &:= \! \pr^2 c_1\!(k_4 \! + \! \pr k_5) \! + \! \pr^2k_1\!(b_3 \!- \! \pr k_4 \! - \! \pr^2 k_5)
			\! + \! (\pr k_2\! + \! \pr^2 k_3)^2\!.
		\end{align*}
	Since $\lim_{\pr \to 0} p_1(\pr)/\pr \!=\! \lim_{\pr \to 0} p_2(\pr)/\pr \!=\! 0$, there exists $\bar{\pr} \!\in\! (0,\bar{\pr}_1)$ such that~\eqref{eq:condition_theorem} is met for all $\pr \!\in\! (0,\bar{\pr})$.
	The proof of~\eqref{eq:result_generic} follows by letting $\psi \!>\! 0$ be the smallest eigenvalue of $\Psi(\pr)$. 
\end{proof}

\subsection{Composite Dynamic Average Consensus}
\label{sec:nested}

In the case of $\agg$-consensus matching~\eqref{eq:nested_agg}, the dynamics~\eqref{eq:local_update_trackers} can be designed as the cascade between two distributed schemes able to solve standard dynamic average consensus problems.
In detail, for all $i \in \set$, given the local aggregation rules $\laggIi$ and $\laggEi$ and the (fixed) variables $\statei$, we consider the cascade 
\begin{subequations}\label{eq:generic_consensus_dynamics}
	\begin{align}
		\hspace{-.2cm}
		\zIitp \!&=\! \trIi(\laggIni(\stateni),\zInit)
		\\
		\hspace{-.2cm}
		\zEitp \!&=\! \trEi(\laggEni(\stateni,\taggIni(\laggIni(\stateni),\zInit)),\zEit),\!
		\label{eq:generic_consensus_dynamics_E}
	\end{align}
\end{subequations}
where $\zIit \!\!\in\!\! \R^{\nzIi}$ and $\zEit \!\!\in\!\! \R^{\nzEi}$ are the states, $\trIi: \R^{\degi\naI} \!\times\! \R^{\sum_{j\in \cN_i}\nzIj} \!\to\! \R^{\nzIi}$ and $\trEi: \R^{\degi\naE} \!\times\! \R^{\sum_{j\in \cN_i}\nzEj} \! \to \! \R^{\nzEi}$ describe their dynamics.
Finally, $\laggIni$ and $\laggEni$ collect the contributions $\laggIj$ and $\laggEj$ of the neighbors of agent $i$ to $\aggI$ and $\aggE$, while $\taggIni$ collects their estimates $\taggIi: \R^{\naI} \times \R^{\nzIi} \to \R^{\naI}$ of $\aggI$, namely
\begin{align*}
	\laggIni(\stateni) &:= \col{\laggIj(\statej)}_{j \in \cN_i}
	\\
	\laggEni(\stateni,v_{\cN_i}) &:= \col{\laggEj(\statej,v_j)}_{j \in \cN_i}
	\\
	\taggIni(\state_{\cN_i},\zni) &:= \col{\taggIj(\laggIj(\state_j),\z_j)}_{j \in \cN_i},
\end{align*}
with $v_j \!\in\! \R^{\naI}$ and $\z_j \!\in\! \R^{\nzIj}$ for all $j \!\!\in\!\! \cN_i$.
We stress that agent $i$ locally approximates $\aggI(\state)$ by using $\taggIi(\laggIi(\statei),\zIi)$ into~\eqref{eq:generic_consensus_dynamics_E}.
By collecting the local updates~\eqref{eq:generic_consensus_dynamics}, we get%
\begin{subequations}\label{eq:generic_consensus_dynamics_global}
	\begin{align}
		\zItp &= \trI(\laggI(\state),\zIt)
		\\
		\zEtp &= \trE(\laggE(\state,\taggI(\laggI(\state),\zIt)),\zEt),
	\end{align}
\end{subequations}
where $\zI := \col{\z_{{\scriptscriptstyle I,1}},\dots,\z_{\ssI,{\scriptscriptstyle N}}} \in \R^{\nzI}$, $\zE := \col{\z_{{\scriptscriptstyle E,1}},\dots,\z_{\ssE,{\scriptscriptstyle N}}} \in \R^{\nzE}$ with $\nzI := \sum_{i=1}^N \nzIi$ and $\nzE := \sum_{i=1}^N \nzEi$, while $\laggI: \R^{\nstate} \to \R^{N\naI}$, $\taggI: \R^{N\naI} \times \R^{\nzI}$, and $\laggE: \R^{N\naE} \times \R^{\nzE} \to \R^{N\naE}$, $\trI: \R^{N\naI} \times \R^{\nzI} \to \R^{\nzI}$, and $\trE: \R^{N\nzE} \times \R^{\nzE} \to \R^{\nzE}$ are defined as 
\begin{align*}
	&\laggI(\state) \!\!:=\!\!\! \begin{bmatrix}\lagg_{{\scriptscriptstyle I,1}}(\state_1)\\
		\vdots\\
		\lagg_{{\scriptscriptstyle I,N}}(\state_N)
	\end{bmatrix}\!\!, \hspace{.025cm} \taggI(\laggI(\state),\zI) \!\!:=\!\!\! \begin{bmatrix}
		\tagg_{{\scriptscriptstyle I,1}}(\lagg_{{\scriptscriptstyle I,1}}(\state_1),\z_{{\scriptscriptstyle I,1}}) 
		\\
		\vdots 
		\\
		\tagg_{{\scriptscriptstyle I,N}}(\lagg_{{\scriptscriptstyle I,N}}(\state_N),\z_{{\scriptscriptstyle I,N}}) 
	\end{bmatrix}
\end{align*}	
\begin{align*}
	&\laggE(\state,\taggI(\laggI(\state),\zI)) := \begin{bmatrix}\lagg_{{\scriptscriptstyle E,1}}(\state_1,\tagg_{\ssI,\scriptscriptstyle{1}}(\lagg_{\ssI,\scriptscriptstyle{1}}(\state_1),\z_{\ssI,\scriptscriptstyle{1}}))\\
		\vdots\\
		\lagg_{\ssE,\scriptscriptstyle{N}}(\state_N,\tagg_{{\scriptscriptstyle I,N}}(\lagg_{\ssI,\scriptscriptstyle{N}}(\state_N),\z_{\ssI,\scriptscriptstyle{N}}))
	\end{bmatrix}\!\!
	\\
	&\trI(\laggI(\state),\zI) := \begin{bmatrix}
		\tr_{\ssI,\scriptscriptstyle{1}}(\lagg_{\ssI,\scriptscriptstyle{1}}(\state_1),\z_{\ssI,\scriptscriptstyle{1}}) 
		\\
		\vdots
		\\
		\tr_{\ssI,\scriptscriptstyle{N}}(\lagg_{\ssI,\scriptscriptstyle{N}}(\state_N),\z_{\ssI,\scriptscriptstyle{N}}) 
	\end{bmatrix}
	\\
	&\trE(v_\ssE,\zE) := \begin{bmatrix}
		\tr_{\ssI,\scriptscriptstyle{1}}(v_{\ssE,\scriptscriptstyle{1}},\z_{\ssE,\scriptscriptstyle{1}}) 
		\\
		\vdots
		\\
		\tr_{\ssI,\scriptscriptstyle{N}}(v_{\ssE,\scriptscriptstyle{N}},\z_{\ssE,\scriptscriptstyle{N}})
	\end{bmatrix}.
\end{align*}
We also introduce $\taggEi: \R^{\naE} \times \R^{\nzEi} \to \R^{\naE}$ that locally approximates $\aggE$.
Then, let $\taggE:\R^{N\naE} \times \R^{\nzE} \to R^{N\naE}$ be 
\begin{align*}
	\taggE(v,w) := \begin{bmatrix}\tagg_{\ssE,\scriptscriptstyle{1}}(v_1,w_1)\T& \hdots& \tagg_{\ssE,\scriptscriptstyle{N}}(v_N,w_N)\T
	\end{bmatrix}\T.
\end{align*}
Given $\uI:=\col{u_{\ssI,\scriptscriptstyle{1}}, \dots, u_{\ssI,\scriptscriptstyle{N}}} \in \R^{N\naI}$ and $\uE:=\col{u_{\ssE,\scriptscriptstyle{1}}, \dots u_{\ssE,\scriptscriptstyle{N}}} \in \R^{N\naE}$ with $u_{\ssI,\scriptscriptstyle{i}}\in \R^{\naI}$ and $u_{\ssE,\scriptscriptstyle{i}} \in \R^{\naE}$ for all $i \in \set$, we consider the ``internal'' auxiliary system
	\begin{align}
		\zItp &= \trI(\uI,\zIt),\label{eq:auxiliary_dynamic_average_consensus_I}
	\end{align}
and the ``external'' one described by
\begin{align}
	\zEtp &= \trE(\uE,\zEt).\label{eq:auxiliary_dynamic_average_consensus_E}
\end{align}
We now give the definition of Linearly Converging Dynamic Average Consensus (\acro/) to formalize the requirements needed by the subsystems of~\eqref{eq:generic_consensus_dynamics_global} (separately considered) to make their cascade fit for the role of the consensus scheme~\eqref{eq:consensus_oriented}.
Given $N$ vectors $u_1,\dots, u_N \!\in\! \R^{\na}$, this definition uses $u \!:=\!\col{u_1,\dots,u_N}\!\in\! \R^{N\na}$ and $u_{\cN_i} \!:=\! \col{u_j}_{j \in \cN_i}\!\in\! \R^{\degi\na}$.
\begin{definition}[\acro/]\label{def:dynamic_average_consensus}
	Let $\zit \in \R^{\nzi}$ and consider
	\begin{align}\label{eq:generic_tracker_local}
		\zitp = \tri(u_{\cN_i},\znit),
	\end{align}
	with $\tri: \R^{\degi\na} \!\times\! \R^{\sum_{j\in\cN_i}^N \nzj} \!\to\! \R^{\nzi}$, and the stacked form 
	\begin{align}\label{eq:generic_tracker_global}
		\ztp = \tr(u,\zt),
	\end{align}
	with $\zt:=\col{\zt_1,\dots,\zt_N}\in \R^{\nz}$ and $\tr(u,\zt) := \col{\tr_1(u_{\cN_1},\zt_{\cN_1}),\dots,\tr_N(u_{\cN_N},\zt_{\cN_N})}\in \R^{\nz}$ with $\nz := \sum_{i=1}^N \nzi$.
	We say that system~\eqref{eq:generic_tracker_global} is a Linearly Converging Algorithm for Dynamic Average Consensus (\acro/) if it satisfies Assumptions~\ref{ass:orthogonality},~\ref{ass:equilibria_consensus}, and~\ref{ass:tracking}.\oprocend
\end{definition}
\begin{proposition}\label{prop:nested_trackers}
	Assume that $\laggI$ and $\laggE$ are Lipschitz continuous with constants $\lipp_{\laggI},\lipp_{\laggE} \!>\! 0$, respectively.
	Assume that~\eqref{eq:auxiliary_dynamic_average_consensus_I} and~\eqref{eq:auxiliary_dynamic_average_consensus_E} are \acro/ with transformation maps $\TTI(\zI) \!\!:=\!\! \begin{bmatrix}\TbI(\zI)\T &\!\TpI(\zI)\T\end{bmatrix}\T$ and $\TTE(\zE) := \begin{bmatrix}\TbE(\zE)\T &\TpE(\zE)\T\end{bmatrix}\T$, invariant sets $\cSbzI$ and $\cSbzE$, equilibrium functions $\pzIeq$ and $\pzEeq$, approximation functions $\taggI$ and $\taggE$, and Lyapunov functions $\VzI$ and $\VzE$, respectively.
	Then, system~\eqref{eq:generic_consensus_dynamics_global} satisfies Assumptions~\ref{ass:orthogonality},~\ref{ass:equilibria_consensus}, and~\ref{ass:tracking} with%
	\begin{subequations}\label{eq:settings}
	\begin{align}
		&\z :=
		\begin{bmatrix}
		   \zI
		   \\
		   \zE
	   \end{bmatrix},\quad \cSbz := \cSbzI \times \cSbzE 		
		\\
		&\Tb(\z) :=  \begin{bmatrix}\TbI(\zI)\\\TbE(\zE)\end{bmatrix}, \quad \Tp(\z) := \begin{bmatrix}\TpI(\zI)\\\TpE(\zE)\end{bmatrix}
		\\
		&
		\pzeq(\state) :=  
		\begin{bmatrix}
			\pzIeq(\state)
			\\
			\pzEeq(\state)
		\end{bmatrix}
		\\
		&\tr(\state,\z) := \begin{bmatrix}
			\trI(\laggI(\state),\zI)
			\\
			\trE(\laggE(\state,\taggI(\laggI(\state),\zI)),\zE)
		\end{bmatrix}
		\\
		&\tagg(\state,\z) := \begin{bmatrix}
			\taggI(\laggI(\state),\zI)
			\\
			\taggE(\laggE(\state,\taggI(\laggI(\state),\zI)),\zE)
		\end{bmatrix}\label{eq:tagg_perturbed_consensus}
		\\
		&\Vz(\tzI,\tzE) := \VzI(\tzI) + \kappa\VzE(\tzE),
		\label{eq:Vz_VzI_VzE}
	\end{align}
	\end{subequations}
	for a sufficiently large $\kappa > 0$.
\end{proposition}
\begin{proof}
	Given $\nbI$, $\nbE$, $\npI$, $\npE \!\in\! \N$ such that $\nzI \!=\! \nbI \!+\! \npI$ and $\nzE \!=\! \nbE \!+\! \npE$, let $\bzI \!\in\! \R^{\nbI}$, $\bzE\!\in\! \R^{\nbE}$, $\pzI \!\in\! \R^{\npI}$, and $\pzE \!\in\! \R^{\npE}$ be defined as 
	\begin{align}
		\begin{bmatrix}
			\bzI 
			\\
			\pzI
		\end{bmatrix} := \TTI(\zI), \quad \begin{bmatrix}
			\bzE
			\\
			\pzE
		\end{bmatrix} := \TTE(\zE),
	\end{align}
	where, by Definition~\ref{def:dynamic_average_consensus}, $\TTI$ and $\TTE$ comply with the generic properties in Assumption~\ref{ass:orthogonality}.
	In detail, in light of the invariance properties of $\cSbzI$ and $\cSbzE$ for~\eqref{eq:auxiliary_dynamic_average_consensus_I} and~\eqref{eq:auxiliary_dynamic_average_consensus_E}, respectively, we exploit~\eqref{eq:orthogonality} in Assumption~\ref{ass:orthogonality} to rewrite the cascade~\eqref{eq:generic_consensus_dynamics_global} as
	\begin{subequations}\label{eq:I_E_decomposed}
		\begin{align}
			\begin{bmatrix}
				\bzItp 
				\\
				\bzEtp
			\end{bmatrix} &= 
			\begin{bmatrix}
				\TbI(\trI(\laggI(\state),\zIt))
				\\
				\TbE(\trE(\laggE(\state,\ptaggI(\state,\pzIt)),\zEt))
			\end{bmatrix}\label{eq:I_E_decomposed_bar}
			\\
			\begin{bmatrix}
				\pzItp 
				\\
				\pzEtp
			\end{bmatrix} &=
			\begin{bmatrix}
				\ptrI(\laggI(\state),\pzIt) 
				\\[.2em]
				\ptrE(\laggE(\state,\ptaggI(\state,\pzIt)),\pzEt) 
			\end{bmatrix}
			,
			\label{eq:I_E_decomposed_perp}
		\end{align}
	\end{subequations}
	for suitable functions $\ptaggI$ and $\ptaggE$ (see~\eqref{eq:ptagg}) and $\ptrI$ and $\ptrE$ (cf.~\eqref{eq:ptr}).
	Further, by Definition~\ref{def:dynamic_average_consensus}, it holds~\eqref{eq:zeq_equilibrium} (cf. Assumption~\ref{ass:equilibria_consensus}),
	namely, there exist $\pzIeq: \R^{N\naI} \to \R^{\npI}$ and $\pzEeq: \R^{N\naE} \to \R^{\npE}$ such that, for all $\uI \in \R^{N\naI}$ and $\uE \in \R^{\naE}$, we guarantee the equilibrium conditions
	\begin{subequations}\label{eq:I_E_equilibrium}
		\begin{align}
			\pzIeq(\uI) &= \ptrI(\uI,\pzIeq(\uI))\label{eq:I_equilibrium}
			\\
			\pzEeq(\uE) &= \ptrE\left(\uE,\pzEeq(\uE)\right).\label{eq:E_equilibrium}
		\end{align}
	\end{subequations}
	By Definition~\ref{def:dynamic_average_consensus}, also the reconstruction conditions~\eqref{eq:zeq_reconstruction} (cf. Assumption~\ref{ass:equilibria_consensus}) are verified, namely it holds
	\begin{subequations}\label{eq:I_E_reconstruction}
		\begin{align}
			\ptaggI(\uI,\pzIeq(\uI)) &= \dfrac{\1}{N}\sum_{i=1}^N \uIi\label{eq:I_reconstruction}
			\\
			\ptaggE(\uE,\pzEeq(\uE)) &=  \dfrac{\1}{N}\sum_{i=1}^N \uEi,\label{eq:E_reconstruction}
		\end{align}
	\end{subequations}
	for all $\uI \in \R^{N\naI}$ and $\uE \in \R^{\naE}$.
	By noting also the Lipschitz conditions ensured by Definition~\ref{def:dynamic_average_consensus} and looking at~\eqref{eq:I_E_decomposed},~\eqref{eq:I_E_equilibrium}, and~\eqref{eq:I_E_reconstruction}, we have that the cascade system~\eqref{eq:generic_consensus_dynamics_global} satisfies the orthogonality, reconstruction, equilibria, and Lipschitz conditions in Assumptions~\ref{ass:orthogonality},~\ref{ass:equilibria_consensus}, and~\ref{ass:tracking} with the settings detailed in~\eqref{eq:settings}.
	Hence, to conclude the proof, we need to show that the Lyapunov function $\Vz$ (cf.~\eqref{eq:Vz_VzI_VzE}) satisfies the conditions detailed in~\eqref{eq:Vz_theorem_distributed} (cf. Assumption~\ref{ass:tracking}) along the cascade system~\eqref{eq:generic_consensus_dynamics_global}.
	Then, we neglect~\eqref{eq:I_E_decomposed_bar} and focus on~\eqref{eq:I_E_decomposed_perp} written in the coordinates $\tzI:= \pzI - \pzIeq(\state)$ and $\tzE := \pzE - \pzEeq(\state)$, namely
	\begin{subequations}\label{eq:I_E_error_coordinates}
		\begin{align}
			\tzItp &= \pttrI(\state,\tzIt)\label{eq:I_error_coordinates}
			\\
			\tzEtp &= \pttrE(\state,\tzIt,\tzEt),\label{eq:E_error_coordinates}
		\end{align}
	\end{subequations}
	where %
	\begin{align*}
		\pttrI(\state,\tzI) &:= \ptrI(\laggI(\state),\tzI + \pzIeq(\state)) - \pzIeq(\state)
		\\
		\pttrE(\state,\tzI,\tzE) &:= - \pzEeq(\state)
		+\trE(\laggE(\state,\ptaggI(\laggI(\state),\tzI + \pzIeq(\state))),\tzE + \pzEeq(\state)).
	\end{align*}
	The Lipschitz properties in Definition~\ref{def:dynamic_average_consensus} and the ones of each $\laggIi$ and $\laggEi$ ensure there exist $\lipp_{{\scriptscriptstyle \pttrI}}, \lipp_{{\scriptscriptstyle \pttrE}}\!>\! 0$ such that%
	\begin{subequations}\label{eq:lipp_ttr}
		\begin{align}
			\norm{\pttrI(\state,\tzI)  -  \pttrI(\state^\prime,\tzI^\prime)} 
			&
			\leq  \lipp_{{\scriptscriptstyle \pttrI}}(\norm{\state   -  \state^\prime}  +  \norm{\tzI  -  \tzI^\prime})
			\\
			\norm{\pttrE(\state,\tzI,\tzE)  -  \pttrE(\state^\prime,\tzI^\prime,\tzE^\prime)}  
			&
			\leq  \lipp_{{\scriptscriptstyle \pttrE}}(\norm{\state  -  \state^\prime}  \! + \!  \norm{\tzI  -  \tzI^\prime} \! + \! \norm{\tzE -  \tzE^\prime}),\label{eq:lipp_ttrE}
		\end{align}
	\end{subequations}
	for all $\state, \state^\prime \in \R^{\nstate}$, $\tzI, \tzI^\prime \in \R^{\nzI}$, and $\tzE, \tzE^\prime \in \R^{\nzE}$.
	By Definition~\ref{def:dynamic_average_consensus}, the two schemes separately considered satisfy Assumption~\ref{ass:tracking}.
	Thus, there exists $\VzI\!:\! \R^{\npI} \!\to\! \R$ such that 
	\begin{subequations}\label{eq:VzI}
		\begin{align}
			&b_{{\scriptscriptstyle I,1}}\norm{\tzI}^2 \leq \VzI(\tzI) \leq b_{{\scriptscriptstyle I,2}}\norm{\tzI}^2\label{eq:VzI_quadratic_bound}
			\\
			&\VzI( \pttrI(\state,\tzI)) \! - \! \VzI(\tzI) \leq - b_{{\scriptscriptstyle I,3}}\norm{\tz}^2\label{eq:VzI_minus}
			\\
			&|\VzI(\tzI) - \Vz(\tzI^\prime)| \leq b_4\norm{\tzI - \tzI^\prime}(\norm{\tzI} + \norm{\tzI^\prime}),
		\end{align}
	\end{subequations}
	for all $\state \!\in \! \cSc$, $\tzI, \tzI^\prime \! \in \! \R^{\npI}$, and some $b_{{\scriptscriptstyle I,1}}$, $b_{{\scriptscriptstyle I,2}}$, $b_{{\scriptscriptstyle I,3}}$, $b_{{\scriptscriptstyle I,4}} \! > \! 0$.
	Analogously, there exists $\VzE: \R^{\npE}\! \to\! \R$ such that%
	\begin{subequations}\label{eq:VzE}
		\begin{align}
			&b_{{\scriptscriptstyle E,1}}\norm{\tzE}^2 \leq \VzE(\tzE) \leq b_{{\scriptscriptstyle E,2}}\norm{\tzE}^2\label{eq:VzE_quadratic_bound}
			\\
			&\VzE(\pttrE(\state,0,\tzE)) \! - \! \VzE(\tzE) \leq - b_{{\scriptscriptstyle E,3}}\norm{\tzE}^2 \label{eq:VzE_minus}
			\\
			&|\VzE(\tzE) \! - \! \VzE(\tzE^\prime)| \! \leq\! b_4\norm{\tzE \! - \! \tzE^\prime}(\norm{\tzE} \! + \! \norm{\tzE^\prime}),\label{eq:VzE_lipp}
		\end{align}
	\end{subequations}
	for all $\state \in \R^{\nstate}$, $\tzE, \tzE^\prime \in \R^{\npI}$, and some $b_{{\scriptscriptstyle E,1}}$, $b_{{\scriptscriptstyle E,2}}$, $b_{{\scriptscriptstyle E,3}}$, $b_{{\scriptscriptstyle E,4}} > 0$,
	We highlight 
	that the inequality~\eqref{eq:VzE_minus} is written by considering $\tzI = 0$ into~\eqref{eq:E_error_coordinates}.
	Now, let us focus on the Lyapunov function $\Vz$ (cf.~\eqref{eq:Vz_VzI_VzE}), where $\kappa$ will be suitably fixed in the next.
	By evaluating $\Delta\Vz(\tzIt,\tzEt):=\Vz(\tzItp,\tzEtp)-\Vz(\tzIt,\tzEt)$ along the trajectories of system~\eqref{eq:I_E_error_coordinates}, we get 
	\begin{align*}
		\Delta \Vz(\tzIt,\tzEt) 
		&= \VzI(\pttrI(\state,\tzIt)) - \VzI(\tzIt) 
		+\kappa\VzE(\pttrE(\state,\tzIt,\tzEt))
		- \kappa\VzE(\tzEt).
	\end{align*}
	We note that the first two terms of the right-hand side can be bounded by resorting to~\eqref{eq:VzI_minus}, thus obtaining 
	\begin{align}
		\Delta \Vz(\tzIt,\tzEt) 
		&\leq -b_{{\scriptscriptstyle I,2}}\norm{\tzIt}^2
		+\kappa\VzE(\pttrE(\state,\tzIt,\tzEt))
		\! - \! \kappa\VzE(\tzEt).\label{eq:Delta_Vz_I_E}
	\end{align}
	As for the second line, the bound~\eqref{eq:VzE_minus} cannot be used since it requires $\pttrE$ evaluated at $\tzIt \!=\! 0$.
	Thus, we add and subtract $\kappa\VzE(\pttrE(\state,0,\tzEt))$ to the right-hand side of~\eqref{eq:Delta_Vz_I_E} and obtain 
	\begin{align}
		\Delta \Vz(\tzIt,\tzEt) 
		&\leq -b_{{\scriptscriptstyle I,2}}\norm{\tzIt}^2
		+\kappa\VzE(\pttrE(\state,0,\tzEt))
		-\kappa\VzE(\tzEt)
		+\kappa\VzE(\pttrE(\state,\tzIt,\tzEt))
		- \kappa\VzE(\pttrE(\state,0,\tzEt))
		\notag\\
		&\stackrel{(a)}{\leq} 
		-b_{{\scriptscriptstyle I,2}}\norm{\tzIt}^2
		-\kappa b_{{\scriptscriptstyle E,2}}\norm{\tzEt}^2
		+\kappa\VzE(\pttrE(\state,\tzIt,\tzEt))
		- \kappa\VzE(\pttrE(\state,0,\tzEt))
		\notag
		\\
		&\stackrel{(b)}{\leq} 
		-b_{{\scriptscriptstyle I,2}}\norm{\tzIt}^2
		-\kappa b_{{\scriptscriptstyle E,2}}\norm{\tzEt}^2
		\notag\\
		&\hspace{.4cm}
		+\kappa b_{{\scriptscriptstyle E,4}}\big\|\pttrE(\state,\tzIt,\tzEt) - \pttrE(\state,0,\tzEt)\big\|\big(\big\|\pttrE(\state,\tzIt,\tzEt)\big\| + \big\|\pttrE(\state,0,\tzEt)\big\|\big)
		\notag\\
		&\stackrel{(c)}{\leq} 
		 -  b_{{\scriptscriptstyle I,2}}\norm{\tzIt}^2
		-\kappa b_{{\scriptscriptstyle E,2}}\norm{\tzEt}^2
		 +  \kappa b_{{\scriptscriptstyle E,4}}\lipp_{{\scriptscriptstyle \pttrE}} \norm{\tzIt}\|\pttrE(\state,\tzIt,\tzEt)\| 
		 \notag\\
		 &\hspace{.4cm}
		  +  \kappa b_{{\scriptscriptstyle E,4}}\lipp_{{\scriptscriptstyle \pttrE}} \norm{\tzIt}\|\pttrE(\state,0,\tzEt)\|,\label{eq:Delta_Vz_I_E_2}
	\end{align}
	where in $(a)$ we apply~\eqref{eq:VzE_minus}, while in $(b)$ we use~\eqref{eq:VzE_lipp}, in $(c)$ we use the Lipschitz continuity of $\pttrE$ (cf.~\eqref{eq:lipp_ttrE}).
	Now, we note that, in light of~\eqref{eq:E_equilibrium}, it holds $\pttrE(\state,0,0) \!=\! 0$ for all $\state \!\in\! \R^{\nstate}$. 
	Thus, we add $\pm\pttrE(\state,0,0)$ into both the terms $\|\pttrE(\state,\tzIt,\tzEt)\|$ and $\|\pttrE(\state,\tzIt,\tzEt)\|$ of~\eqref{eq:Delta_Vz_I_E_2} and, by using the Lipschitz continuity of $\pttrE$ (cf.~\eqref{eq:lipp_ttrE}), we get 
	\begin{align}
		&\Delta \Vz(\tzIt,\tzEt) 
		\leq -\begin{bmatrix}
			\|\tzIt\|
			\\
			\|\tzEt\|
		\end{bmatrix}\T \tilde{Q}(\kappa)\begin{bmatrix}
			\|\tzIt\|
			\\
			\|\tzEt\|
			\end{bmatrix},
		\label{eq:Delta_Vz_I_E_3}
	\end{align}
	where we introduced $\tilde{Q}(\kappa) = \tilde{Q}(\kappa)\T \in \R^{2 \times 2}$ defined as 
	\begin{align*}
		\tilde{Q}(\kappa) := \begin{bmatrix}
			b_{{\scriptscriptstyle I,2}} - \kappa b_{{\scriptscriptstyle E,4}}\lipp_{{\scriptscriptstyle \pttrE}}^2& -\kappa b_{{\scriptscriptstyle E,4}}\lipp_{{\scriptscriptstyle \pttrE}}^2
			\\
			-\kappa b_{{\scriptscriptstyle E,4}}\lipp_{{\scriptscriptstyle \pttrE}}^2& \kappa b_{{\scriptscriptstyle E,2}}
		\end{bmatrix}.
	\end{align*}
	We know that $\tilde{Q}(\kappa) = \tilde{Q}(\kappa)\T > 0$ if and only if 
	\begin{align*}
		\begin{cases}
			b_{{\scriptscriptstyle I,2}} > \kappa b_{{\scriptscriptstyle E,4}}\lipp_{{\scriptscriptstyle \pttrE}}^2 
			\\
			\kappa b_{{\scriptscriptstyle I,2}}b_{{\scriptscriptstyle E,2}}  > \kappa^2  \left(b_{{\scriptscriptstyle E,4}}^2\lipp_{{\scriptscriptstyle \pttrE}}^4 + b_{{\scriptscriptstyle E,2}}b_{{\scriptscriptstyle E,4}}\lipp_{{\scriptscriptstyle \pttrE}}^2\right).
		\end{cases}
	\end{align*}
	Then, by setting 
		$\kappa < \min\left\{\tfrac{b_{{\scriptscriptstyle I,2}}}{b_{{\scriptscriptstyle E,4}}\lipp_{{\scriptscriptstyle \pttrE}}^2},\tfrac{b_{{\scriptscriptstyle I,2}}b_{{\scriptscriptstyle E,2}}}{b_{{\scriptscriptstyle E,4}}^2\lipp_{{\scriptscriptstyle \pttrE}}^4 + b_{{\scriptscriptstyle E,2}}b_{{\scriptscriptstyle E,4}}\lipp_{{\scriptscriptstyle \pttrE}}^2}\right\}$
	and denoting with $\tilde{q} > 0$ the smallest eigenvalue of $\tilde{Q}(\kappa)$, we bound the right-hand side of~\eqref{eq:Delta_Vz_I_E_3} as 
	\begin{align*}
		\Delta \Vz(\tzItp,\tzEtp) \leq - \tilde{q}(\|\tzIt\|^2 + \|\tzEt\|^2),
	\end{align*}
	which
	concludes the proof.
\end{proof}

\end{document}